\documentclass[11pt,fullpage]{article}
\usepackage{amsmath}
\usepackage{amsfonts,amssymb}
\usepackage{amsthm}
\usepackage{graphics,graphicx,pspicture,graphpap}
\usepackage{hyperref}
\usepackage[right = 2.5cm, left=2.5cm, top = 2.5cm, bottom =2.5cm]{geometry}
\usepackage{color}
\usepackage{pstricks,pst-node,pst-text,pst-3d}

\numberwithin{equation}{section}

\newcommand{\Real}{\mathbb R}
\newcommand{\norm}[1]{\|#1\|}
\newcommand{\abs}[1]{\left\vert#1\right\vert}
\newcommand{\set}[1]{\left\{#1\right\}}

\newcommand{\grad}{\nabla}

\newcommand{\K}{\mathcal{K}}

\newcommand{\F}{\mathcal{F}}

\newcommand{\supp}{\textup{supp}\,}

\newcommand{\e}{\epsilon}
\newcommand{\jap}[1]{\langle #1 \rangle} 
\newcommand{\RealPart}{\textup{Re}\,}
\newcommand{\ImPart}{\textup{Im}\,}

\newtheorem{theorem}{Theorem}
\theoremstyle{definition}
\newtheorem{remark}{Remark}

\theoremstyle{lemma}

\newtheorem{proposition}{Proposition}
\newtheorem{corollary}{Corollary}
\theoremstyle{definition}
\newtheorem{definition}{Definition}
\theoremstyle{lemma}
\newtheorem{lemma}{Lemma}

\begin{document}

\title{Global Existence and Finite Time Blow-Up for Critical Patlak-Keller-Segel Models with Inhomogeneous Diffusion} 

\author{Jacob Bedrossian\footnote{\textit{jacob@cims.nyu.edu}, New York University, Courant Institute of Mathematical Sciences.}, 
 Inwon C. Kim \footnote{\textit{ikim@math.ucla.edu}, University of California-Los Angeles, Department of Mathematics.}}

\date{}

\maketitle

\begin{abstract}
The $L^1$-critical parabolic-elliptic Patlak-Keller-Segel system is a classical model of chemotactic aggregation in micro-organisms well-known to have critical mass phenomena \cite{BlanchetEJDE06,Blanchet09}.
In this paper we study this critical mass phenomenon in the context of Patlak-Keller-Segel models with spatially varying diffusivity of the chemo-attractant. 
The primary issue is how, if possible, one localizes the presence of the inhomogeneity in the nonlocal term.
We also provide new blow-up results for critical homogeneous problems with nonlinear diffusion, showing that there exist blow-up solutions with arbitrarily large (positive) initial free energy.   
For several non-trivial technical reasons, which we discuss in more detail below, we work in dimensions three and higher
where $L^1$-critical variants of the PKS have porous media-type nonlinear diffusion on the organism density, resulting in finite speed of propagation and simplified functional inequalities.   
\end{abstract}

\section{Introduction} 
The most widely studied mathematical models of nonlocal aggregation phenomena are the parabolic-elliptic Patlak-Keller-Segel (PKS) models, originally introduced to study the chemotaxis of microorganisms \cite{Patlak,KS,Hortsmann,HandP}. 
In this paper we consider $L^1$-critical cases of the form,
\begin{equation} \label{def:ADD} 
\left\{
\begin{array}{l}
  u_t + \grad \cdot (u \grad c) = \Delta u^{2-2/d} \\
  -\grad \cdot (a(x)\grad c) + \gamma(x)c = u \\  
  u(0,x) = u_0(x) \in L_+^1(\Real^d;(1+\abs{x}^2)dx)\cap L^\infty(\Real^d), \;\; d \geq 3,
\end{array}
\right.
\end{equation}
where  $L_+^1(\Real^d;\mu) := \set{f \in L^1(\Real^d;\mu): f \geq 0}$. We define $m := 2 - 2/d$ which is the nonlinear diffusion power for which the system is $L^1$-critical. 
In this context, $L^1$-critical refers to the approximate balance of the opposing forces of diffusion and aggregation in the limit of $L^1$ concentration, indicating that there must be a non-zero, but finite, amount of mass concentration at any possible blow-up.
 As weak solutions to \eqref{def:ADD} conserve mass, we will henceforth refer to $\norm{u_0}_1 = \norm{u(t)}_1 = M(u)$.  
For all our work, we assume that $a(x)$ is strictly positive which ensures the PDE for $c$ to be uniformly elliptic. 
The model \eqref{def:ADD} is a generalization
of the classical parabolic-elliptic 2D PKS model which has received considerable attention over the years (see the review \cite{Hortsmann} and \cite{JagerLuckhaus92,Dolbeault04,BlanchetEJDE06,Blanchet08,BlanchetCarlenCarrillo10}).
For aggregation occurring in three dimensional environments, the nonlinear diffusion models an over-crowding phenomenon which retains the $L^1$ criticality, unlike the corresponding supercritical model with linear diffusion.
It is well-known that all such $L^1$ critical models exhibit critical mass phenomena: there exists some $M_c > 0$ such that if $M(u) < M_c$ then every solution exists globally, and if $M(u) > M_c$ then there exist solutions which blow up in finite time (see for instance \cite{Nagai95,BlanchetEJDE06,Blanchet08,Blanchet09,BRB10}).
We refer to the special case of $a(x) \equiv 1$, $\gamma(x) = 0$ as the \emph{scale-invariant problem}, because solutions are invariant under an $L^1$ scaling in space and this scaling symmetry plays a fundamental role in the global theory.

\vspace{10pt}

In this paper we estimate the critical mass, and under certain restrictions, show that this estimate is sharp for the PKS model \eqref{def:ADD} with spatially variable coefficients in the chemo-attractant PDE. 
The study of spatially variable coefficients raises questions which are both mathematically interesting and relevant for biological applications where chemotaxis occurs in a spatially inhomogeneous medium. 
We introduce inhomogeneity specifically to the nonlocal nonlinearity, since it is not obvious that the effect on the qualitative behavior should be localized.
One could also consider inhomogeneous mobility for the transport of $u$, but this does not add any new interesting complications, so for the sake of clarity we leave this generalization out (see Remark \ref{rmk:Mobility})

\vspace{10pt}

For many problems which are scale-invariant, determining critical thresholds is straightforward with a classical procedure: sharp functional inequalities to prove global existence below the threshold and virial methods to prove blow-up above (see more discussion below). 
As mentioned above, inhomogeneous critical problems can introduce new challenges and change the qualitative behavior of models. The inhomogeneity
will alter the behavior of the solution at all length-scales differently and well-established methods for homogeneous problems can potentially break down.
When proving global existence, the idea is to rely on an approximate scale-invariance which is recovered in the limit of infinitesimal length-scales. 
Naturally, the challenge in this kind of localization lies in the control of the nonlocal error.
On the other hand,  this viewpoint does not help at all when studying blow-up dynamics, as one needs to exploit some monotonicity which persists as the solution blows up and is valid for nonlocal, macroscopic interactions. 
Most approximations, including the available results involving mass comparison, will introduce unbounded errors as the solution blows up, making it challenging or impossible to apply straightforward monotonicity estimates. This issue makes the blow-up problem above critical mass significantly more subtle than the problem of global existence below critical mass. 
\vspace{10pt}

For the inhomogeneous models, our results state that blow-up is spatially localized: it only depends on the local (minimum) value of $a(x)$. 
We estimate the critical mass to be given by
\begin{equation*}
M_c = \left( \frac{2 \inf_{x \in \Real^d} a(x)}{(m-1)C_\star c_d} \right)^{d/2},
\end{equation*}
where $C_\star$ is the optimal constant in the Hardy-Littlewood-Sobolev inequality discussed in \cite{Blanchet09} (see \eqref{ineq:NHLS} below) and $c_d$ is the normalization constant in the Newtonian potential, given explicitly below in \eqref{def:cd}. 
We prove that if $M(u) < M_c$ then the solution is global and uniformly bounded in $L^\infty((0,\infty) \times \Real^d)$ (Theorem \ref{thm:GE}).
Further, we prove a stronger blow-up criteria which is also valid in supercritical mass regimes and shows that the critical mass is truly dependent on space, which accurately confirms the localized nature of the blow-up (Theorem \ref{thm:Concentration}). 
More precisely, we show that in order for a solution to blow-up at a time $T_\star \in (0,\infty]$, for every sequence of times $t_k$ there exists a subsequence (not relabeled) and a sequence of points $x_k$ such that 
\begin{equation*}
\limsup_{r \rightarrow 0^+} \limsup_{k \rightarrow \infty} \int_{\abs{x - x_k} < r} u(t_k,x) dx \geq \liminf_{k \rightarrow \infty}\left( \frac{2 a(x_k)}{(m-1)C_\star c_d} \right)^{d/2}. 
\end{equation*}  
Combining this with mass comparison methods, we prove that radially symmetric solutions with $M(u) = M_c$ are global and uniformly bounded. 
Note, the uniform bound is in contrast to the infinite time aggregation exhibited in $\Real^2$ \cite{Blanchet08} (see more discussion below).   
In the direction of finite time blow-up we prove that if $\gamma = 0$ and $a(x)$ is radially symmetric and monotone non-decreasing in a neighborhood of the origin, then for all $M > M_c$ we may construct a solution with $M(u) = M$ which blows up in finite time (Theorem \ref{thm:blowup}). 
See the discussion below for the proof of the blow-up result and the relevant construction of barriers.
\vspace{10pt}

A key quantity in the study of \eqref{def:ADD} is the dissipated \emph{free energy}, given by 
\begin{equation}
\F(u) = \frac{1}{m-1}\int u^m(x) dx - \frac{1}{2}\int u(x)c(x) dx. \label{def:F}
\end{equation}
The first term is usually referred to as the \emph{entropy} and the latter term is referred to as the \emph{interaction energy} or \emph{potential energy}. 
Formally, \eqref{def:ADD} is a gradient flow with respect to the Euclidean Wasserstein distance for \eqref{def:F} (see e.g. \cite{BlanchetCalvezCarrillo08}), but this will not be relevant for our work (however, see \cite{BlanchetCarlenCarrillo10} for work on the threshold problem with linear diffusion where this structure is the key tool). 

\vspace{10pt}

We consider the question of global existence first. 
It is well-known that solutions to systems such as \eqref{def:ADD} exist as long as they remain equi-integrable \cite{JagerLuckhaus92,CalvezCarrillo06,Blanchet09,BRB10}.  
The use of sharp functional inequalities to identify when a mixed-sign energy such as \eqref{def:F} is coercive
is classical, for example \cite{Weinstein83,Witelski04,BlanchetEJDE06,Blanchet09,BRB10}. 
In the context of classical PKS and similar models, sharp Hardy-Littlewood-Sobolev inequalities prove that when $M(u) < M_c$ the free energy uniformly controls the entropy (in this case the $L^m$ norm), which rules out any loss of equi-integrability (see \cite{Dolbeault04,BlanchetEJDE06,CalvezCarrillo06,Blanchet09,BRB10}).     
In \cite{BRB10}, it was shown for relevant systems with $c = \K \ast u$, for interaction kernels $\K$ in a general class, that the critical mass is governed only by the asymptotic expansion of $\K$ at the origin.
There, the asymptotic expansion describes the singular contribution to the potential energy $\int u c dx$ and the remaining error is `subcritical' in the sense that it can be controlled by norms weaker than the entropy. 
This is, in spirit, the approach taken here to prove Theorems \ref{thm:GE} and \ref{thm:Concentration} in \S\ref{sec:GE}, however a different method must be applied for \eqref{def:ADD}, since in our case $c$ is not given by a convolution. 
Instead, we use a pseudo-differential operator to explicitly evaluate the leading order contribution to the potential energy as a quadratic form and show that the remainder is subcritical.
This allows to refine the classical methods and estimate the critical mass. 
Additionally, the critical mass blow-up criterion (Theorem \ref{thm:Concentration}) is deduced by refining and iterating the arguments of \cite{Blanchet08} (see below for more information). A crucial new tool here is a geometric covering lemma which, combined with the concentration compactness principle \cite{LionsCC84}, allows to decompose blow-up solutions into isolated regions in which a localized energy argument can be applied.

\vspace{10pt}

Conventionally, the simplest way of proving blow-up for systems such as \eqref{def:ADD} above the critical threshold is the well-known virial method, used in, for example \cite{Nagai95,SugiyamaDIE06,Blanchet09,Senba02}.
Applied to the scale-invariant problem in $d\geq 3$, this method consists of two steps (see e.g. \cite{Blanchet09}). 
First, one uses supercritical mass and a scaling argument to construct initial data with negative free energy. 
Second, one shows that negative free energy would force the second moment to zero in finite time.  
Although we do not claim that this cannot be done for systems such as \eqref{def:ADD}, approximating the nonlocal term with pseudo-differential operators as done in the global existence results introduces error terms for which there seems to be no obvious way to control into blow-up.
Hence, we instead use a different finite time blow-up proof which is able to treat the local and nonlocal properties of the solution with sufficient precision. 

\vspace{10pt}

The proof of finite time blow-up is based on a maximum-principle type argument on the mass distribution of solutions. 
The use of mass comparison principles for models such as \eqref{def:ADD} is by now classical \cite{DiazNagai95,DiazNagaiRakotoson98,Senba02,JagerLuckhaus92,BilerKarchEtAl06,KimYao11,CieslakWinkler08}. However, our work is the first one (to our knowledge) which treats blow-up phenomena with mass arbitrarily close to the critical threshold. This is possible due to the precision of the barrier. 
Specifically, we compare solutions of \eqref{def:ADD} to the extremizers of the sharp Hardy-Littlewood-Sobolev inequality which governs the critical mass, which are also stationary solutions of the scale-invariant PDE (Theorem \ref{thm:SharpHLS}). More details are discussed in Section \S\ref{sec:FTBU}. 

\vspace{10pt}

Due to the nature of the near-critical regime we address, the mass comparison argument we adopt is rather delicate and depends on certain regularity properties of solutions as well as the barrier. Among other complications, the degenerate diffusion in \eqref{def:ADD} implies that classical regularity is not available everywhere and the support of solutions move at a finite speed.
In order to deal with the generic presence of the free boundary,  we utilize the viscosity solution theory developed for degenerate diffusion equations with drift (see \cite{KL})  to prove that strictly positive solutions remain positive until blow-up (see Appendix). Then we need to make a careful approximation argument with smooth, strictly positive solutions, as in \cite{KimYao11}.  We add that here extra care must be taken due to the finite time blow-up (Lemma \ref{lem:full_approx}).

\vspace{10pt}

With some modification, this blow-up proof can also provide estimates from above on how quickly mass concentrates, providing also a lower bound on the blow-up time by use of a mass supersolution, which is perhaps a more common use of mass comparison principles. Indeed, this is how the global existence at critical mass is proved for radial solutions (see below).
More interestingly, when applied to the scale-invariant problem, our method yields blow-up for a class of initial data that the results in \cite{Blanchet09}, obtained via a virial method, do not cover (see Theorem \ref{cor:HomogProblem} below).
In particular, we exhibit radially symmetric solutions with initially positive free energy that concentrate in finite time.  
This seems to indicate that the approach we employ could potentially provide different kinds of blow-up results elsewhere, even when straightforward virial methods can be applied. 

\vspace{10pt}

In the last section of the paper we prove that under the assumption of radial symmetry, solutions with critical mass $M(u) = M_c$ exist globally and are uniformly bounded. 
Unlike the homogeneous case, one cannot expect radial monotonicity to hold (even when $a(x)$ is radial), which means that the construction of a mass supersolution is not enough to deduce such a result without additional information.  
On the other hand, since the solutions have exactly critical mass, a simple corollary of Theorem \ref{thm:Concentration} below is that in order to blow up, such threshold solutions must concentrate all of the mass into a single point where the minimum of $a(x)$ is achieved.
Note that Theorem \ref{thm:Concentration} does not imply such rigid behavior at blow-up for solutions with larger than critical mass.
Mass comparison methods similar to those employed to prove finite time blow-up and the arguments in
\cite{BilerKarchEtAl06,Yao11,Senba02,KimYao11} are used to show that this concentration cannot occur in finite or infinite time.   
Hence, the solution must be global and uniformly bounded.  

\vspace{10pt}

We restrict ourselves to the case $\Real^d$ for $d \geq 3$ for several reasons. 
Firstly, note that the $d \geq 3$ case is qualitatively different than the $d = 2$ cases aside from the 
finite speed of propagation effects introduced by the nonlinear diffusion. 
For example, the behavior at the critical mass is different (\cite{Blanchet08,BlanchetCarlenCarrillo10} versus \cite{Blanchet09}) as well as the blow-up dynamics \cite{HerreroVelazquez96,HerreroMedinaVelazquez98,BlanchetLaurencot09,Blanchet11}.
To treat the $\Real^2$ case with our methods would require more sophisticated analysis techniques in the global existence argument (see Remark \ref{rmk:R2inGE} in \S\ref{sec:GE} for a discussion), as errors introduced by the pseudo-differential approximations must be controllable in terms of the Boltzmann entropy. 
There are also advantages in the blow-up argument due to the spatial localization provided by the finite speed of propagation in $d \geq 3$.
Despite these potential difficulties, we expect analogous results to hold in $\Real^2$.

\subsubsection{Notation}

In what follows, we denote $\norm{u}_p := \norm{u}_{L^p(\Real^d)}$ where $L^p(\Real^d) := L^p$ is the standard Lebesgue space. 
We define the homogeneous and inhomogeneous Sobolev spaces as the closure of the Schwartz space under the norms $\norm{u}^2_{\dot{H}^s} := \int \abs{\xi}^{2s}\abs{\hat{u}(\xi)}^2 d\xi$ and $\norm{u}^2_{H^s} := \int (1 + \abs{\xi}^2)^{s}\abs{\hat{u}(\xi)}^2 d\xi$, where $\hat{u}(\xi)$ denotes the Fourier transform of $u(x)$.  
We will often suppress the dependencies of functions on space and/or time to enhance readability. 
The standard characteristic function for some $S \subset \Real^d$ is denoted $\mathbf{1}_{S}$ and we denote the ball $B_R(x_0) := \set{x \in \Real^d : \abs{x - x_0} < R}$.
In addition, we use $\int f dx := \int_{\Real^d} f dx$, and only indicate the domain of integration 
where it differs from $\Real^d$. 
We use $\mathcal{N}$ to denote the Newtonian potential: 
\begin{equation*}
\mathcal{N}(x) = 
\left\{
\begin{array}{ll}
  \frac{1}{2\pi}\log \abs{x} & d = 2 \\ 
  \frac{\Gamma(d/2 + 1)}{d(d-2)\pi^{d/2}}\abs{x-y}^{2-d} & d \geq 3. 
\end{array}
\right.
\end{equation*}
In formulae we use the notation $C(p,M,..)$ to denote a generic constant which depends on parameters $p,M,...$, which may be different from line to line or term to term in the same formula. 
In general, these constants will depend on more parameters than those listed, for instance those which are fixed by the problem, such as $a(x)$ and the dimension, but these dependencies are usually suppressed.
We use the notation $f \lesssim_{p,k,...} g$ to denote $f \leq C(p,k,..)g$ where again, dependencies that are not relevant are sometimes suppressed.

\subsection{Background} 
The local theory for \eqref{def:ADD} is studied in \cite{BR11}. 
Here we simply discuss the results of that work, which follows closely the work of \cite{BlanchetEJDE06,SugiyamaDIE06,BertozziBrandman10,BertozziSlepcev10,BRB10}. We begin with the definition of weak solution, which is stronger than the concept of distribution solutions. 
The main purpose of this definition is to ensure that weak solutions are unique. 

\begin{definition}[Weak Solution]
A function $u(t,x):[0,T] \times \Real^d \rightarrow [0,\infty)$ is a weak solution of \eqref{def:ADD} if $u\in L^\infty((0,T)\times\Real^d)\cap L^\infty(0,T,L^1(\Real^d))$, $u^m\in L^2(0,T,\dot{H}^1(\Real^d))$, $u\grad c \in L^2((0,T)\times\Real^d)$, $u_t\in L^2(0,T,\dot{H}^{-1}(\Real^d))$, and for all test functions $\phi \in \dot{H}^{1}(\Real^d)$ for a.e $t\in [0,T]$, 
\begin{align}\label{WF1}
\left\langle u_t(t), \phi\right\rangle = \int \left(-\nabla u^m(t) + u(t)\grad c(t)\right)\cdot \nabla \phi\; dx, 
\end{align}
where $c(t)$ is the strong solution to the PDE $-\nabla \cdot (a(x)\nabla c(t)) + \gamma(x) c(t) = u(t)$ which vanishes at infinity.
\end{definition}

We state the following theorem summarizing the local theory of \eqref{def:ADD}, developed in \cite{BR11} as well as \cite{SugiyamaDIE06,Blanchet09,BRB10}. 
\begin{theorem} [Local Existence and Uniqueness] \label{thm:loc_theory}
Let $d \geq 3$, let $a(x) \in C^1$ be strictly positive such that both $a(x)$ and $\grad a(x)$ are bounded, let $\gamma(x) \in L^\infty$ be non-negative and $u_0 \in L^1_+(\Real^d;(1 + \abs{x}^2)dx) \cap L^\infty(\Real^d)$. Then there exists a maximal $T_+(u_0) > 0$ and a unique weak solution $u(t)$ to \eqref{def:ADD} which satisfies $u \in C([0,T];L_+^1(\Real^d;(1+\abs{x}^2)dx))\cap L^\infty((0,T);L^\infty(\Real^d))$ for all $T < T_+(u_0)$ and $u(0) = u_0$. Additionally, $\F(u_0) < \infty$ and we have the energy dissipation inequality, 
\begin{equation}
\F(u(t)) + \int_0^t \int u(s)\abs{\grad \frac{m}{m-1}u^{m-1}(s) - \grad c(s)}^2 dx ds \leq \F(u_0). \label{ineq:EnrDiss}
\end{equation}
We also have the continuation criterion: if 
\begin{equation*}
\lim_{k \rightarrow \infty} \sup_{t \in [0,T_+(u_0))} \norm{ (u(t)- k)_+}_1 = 0, 
\end{equation*}
then necessarily $T_+(u_0) = \infty$ and $u(t,x) \in L^\infty(\Real^+ \times \Real^d)$. 
\end{theorem}

The results regarding the critical mass in the constant-coefficient case are summarized in the following theorem.
  
\begin{theorem}[Critical Mass \cite{Blanchet09,BRB10}] \label{thm:critmass}
Suppose $a(x) = a$ and $\gamma(x) = \gamma$ are both constants. 
Then the sharp critical mass satisfies, 
\begin{equation*}
M_c = \left( \frac{2a}{(m-1)C_\star c_d} \right)^{d/2},  
\end{equation*}
and if $u(t,x)$ is a weak solution to \eqref{def:ADD} with $M(u_0) < M_c$, then $u(t,x)$ exists globally, e.g. $T_+(u_0) = \infty$, and we have $u(t,x) \in L^\infty(\Real^+ \times \Real^d)$. Conversely for all $M > M_c$ there exists a solution to \eqref{def:ADD} which blows up in finite time with $M(u_0) = M$. 
Here $C_\star$ is the optimal constant in the Hardy-Littlewood-Sobolev inequality \eqref{ineq:NHLS} below and $c_d$ is the normalization factor in the Newtonian potential:
\begin{equation}
c_d := \frac{\Gamma\left(\frac{d}{2}+1\right)}{d(d-2)\pi^{d/2}}. \label{def:cd}
\end{equation}
\end{theorem}

In \cite{Blanchet09}, Blanchet et. al. exhibit a unique family of stationary solutions to the scale-invariant problem, based on which we will construct barriers used in the proof of finite time blow-up in section 3.

\begin{theorem}[Stationary Solutions to Scale-Invariant Problem \cite{Blanchet09}] \label{thm:V}
There exists a non-negative, radially symmetric, non-increasing function $V(x)$ supported in the ball of radius one with $\norm{V}_1 = \left( \frac{2}{(m-1)C_\star c_d} \right)^{d/2}$ which is the unique solution (up to $L^1$ scaling and translation) of
\begin{equation*}
\Delta V^m = \grad \cdot ( V \grad \mathcal{N} \ast V). 
\end{equation*}
\end{theorem}

\begin{remark} \label{rmk:RescaledV}
Note that if $a > 0$ and $\tilde{V} = a^{d/2} V$, then 
\begin{equation*}
\Delta \tilde{V}^{2-2/d} = \frac{1}{a}\grad \cdot (\tilde{V} \grad \mathcal{N} \ast \tilde{V})
\end{equation*} 
and in light of Theorem \ref{thm:critmass} above, $\tilde{V}$ are the unique (up to $L^1$ scaling and translation) stationary solutions to the problem 
\begin{equation*} 
\left\{
\begin{array}{l}
  u_t + \grad \cdot (u \grad c) = \Delta u^{2-2/d} \\
  -a\Delta c = u.   
\end{array}
\right.
\end{equation*}
\end{remark}

In \cite{Blanchet09}, it is also shown that these stationary solutions are the unique extremals of the following Hardy-Littlewood-Sobolev type inequality. 

\begin{theorem}[Sharp Hardy-Littlewood-Sobolev Inequality \cite{Blanchet09}] \label{thm:SharpHLS}
There exists some optimal $C_\star > 0$ depending only on the dimension such that for all $f \in L_+^1 \cap L^{m}$, 
\begin{equation}
\int\int f(x)f(y)\abs{x-y}^{2-d} dy dx  \leq C_\star \norm{f}_1^{2-m}\norm{f}_m^m, \label{ineq:NHLS}  
\end{equation}
and equality is achieved if and only if there exists $\alpha_0 \in \Real$, $x_0 \in \Real^d, \; \lambda_0 \in (0,\infty)$ such that 
\begin{equation*}
f(x) = \frac{\alpha_0}{\lambda_0^d}V \left( \frac{x-x_0}{\lambda_0} \right). 
\end{equation*}
\end{theorem}

We will also use the following more general inequality, although we have no need of the sharp constant,   
\begin{equation} 
\int\int f(x) g(y) K(x-y) dx dy \lesssim_{p,q,t} \norm{f}_p \norm{g}_q \norm{K}_{L^{t,\infty}}. \label{ineq:Gconvo}
\end{equation} 
Here, $\norm{\cdot}_{L^{t,\infty}}$ denotes the weak $L^t$ space. 

\subsection{Summary of Results} \label{sec:Summary}

We now state the main results. The higher regularity on $a(x)$ and $\gamma(x)$ is assumed so we may use standard symbol classes. It is likely that significantly less regularity is required for the same results. 

\medskip

\begin{theorem}[Global Existence at Subcritical Mass] \label{thm:GE}
Let $d\geq 3$ and $a(x)\in C^\infty(\mathbb{R}^d)$ be strictly positive and $\gamma(x) \in C^\infty(\Real^d)$ be non-negative such that $D^\alpha a$ and $D^\alpha \gamma$ are bounded for all multi-indices $\alpha$.
Then we may estimate the critical mass as  
\begin{equation}
M_c = \left( \frac{2 \inf_{x \in \Real^d} a(x)}{(m-1)C_\star c_d} \right)^{d/2},  \label{def:critmass}
\end{equation} 
and any weak solution $u(t)$ to \eqref{def:ADD} with $M(u) < M_c$ exists globally and $u(t) \in L^\infty(\Real^+ \times \Real^d)$. 
\end{theorem}
\begin{remark}
Using the methods of \cite{CalvezCarrillo06,BRB10}, Theorem \ref{thm:GE} can be extended to cover more general filtration equation-type nonlinear diffusion on the RHS of the PDE for $u(t)$ (e.g. $\Delta A(u)$, $A \in C^1$ and non-decreasing with $0 < \lim\inf_{z \rightarrow \infty}A^\prime(z)z^{2/d-1} < \infty$, instead of simply $\Delta u^{2-2/d}$).  
\end{remark}

The next result is a significant refinement of the previous theorem, providing a strong blow-up criterion for solutions of 
arbitrary mass which requires them to concentrate a specific amount of mass at blow-up time. This theorem indicates that blow-up is indeed necessarily a localized phenomenon and ultimately does not depend on the global or average properties of $a(x)$. 
Theorem \ref{thm:GE} is a direct corollary of this result, however we have chosen to state and prove them separately as the proof of Theorem \ref{thm:GE} is elementary once the proper machinery is in place whereas the proof of Theorem \ref{thm:Concentration} is significantly more technical.  
The methods we employ to prove Theorem \ref{thm:Concentration} rest on a careful estimate of the free energy and are based on iterating a refined version of arguments found in \cite{Blanchet08}.
In order to do so, we employ a delicate geometric decomposition and the concentration compactness principle \cite{LionsCC84,Lions85p1,Lions85p2}. 
The proof of \cite{Blanchet08} is also a type of concentration compactness, and Theorem \ref{thm:Concentration} represents the natural extension to more sophisticated concentration compactness methods which had not (to our knowledge) been previously used in the study of Patlak-Keller-Segel models. 

\begin{theorem}\label{thm:Concentration}
Let $d\geq 3$ and $a(x)\in C^\infty(\mathbb{R}^d)$ be strictly positive and $\gamma(x) \in C^\infty(\Real^d)$ be non-negative such that $D^\alpha a$ and $D^\alpha \gamma$ are bounded for all multi-indices $\alpha$.
Let $u(t)$ be a solution to \eqref{def:ADD} which blows up at $T_+(u_0) \in (0,\infty]$. 
For any sequence of times $t_k \nearrow T_+$ we may extract a subsequence (not relabeled) and find a sequence of points $x_k \in \Real^d$ such that 
\begin{equation}
\limsup_{r \rightarrow 0} \limsup_{k \rightarrow \infty}\int_{B(x_k,r)}u(t_k,x) dx \geq  \left( \frac{2}{(m-1)C_\star c_d} \right)^{d/2}\liminf_{k \rightarrow \infty} a(x_k)^{d/2}. \label{ineq:thmCon}
\end{equation}
\end{theorem} 

\begin{remark}
For any sequence of times $t_k \nearrow T_+$ we may extract a subsequence and a finite Borel measure $n^\star$ such that $u(t_k)dx \rightharpoonup^\star n^\star$. 
If $x_k$ is bounded then we may extract a further subsequence such that $\lim_{k \rightarrow \infty} x_k = x_0$. 
In this case, \eqref{ineq:thmCon} implies
\begin{equation*}
n^\star( \set{x_0} ) \geq \left( \frac{2 a(x_0)}{(m-1)C_\star c_d} \right)^{d/2}. 
\end{equation*}
In the homogeneous case, $\{x_k\}$ is always bounded due to the standard estimate on the second moment. 
In the case of \eqref{def:ADD}, we suspect it may be possible that $\{x_k\}$ be unbounded (i.e. mass will escape to infinity) if the infimum of $a(x)$ lies at infinity. Although if $T_+ < \infty$ we expect (but cannot confirm) that $\{x_k\}$ are always bounded.
\end{remark} 

\begin{remark}
For the homogeneous problem in 2D an analogous, but stronger, result on the Torus by Senba and Suzuki \cite{Senba02} shows that at blow-up time the measure $n^\star$ consists only of atomic parts with at least critical mass and an absolutely continuous part which is smooth away from the concentrations (see also the $\epsilon$-regularity results \cite{Sugiyama09,Sugiyama10}). However, the methods of \cite{Senba02,Sugiyama09,Sugiyama10} are completely different from ours, which are purely variational.
\end{remark}

For solutions with precisely critical mass, called threshold solutions, Theorem \ref{thm:Concentration} implies a very specific blow-up structure: in order for a threshold solution to blow up, it must concentrate all of the mass into a single point where the minimum of $a(x)$ is achieved. For radially symmetric threshold solutions, this is enough to allow us to prove global existence and uniform boundedness (Theorem \ref{thm:GE_CritMass}).
This is in contrast to the $\Real^2$ case, where threshold solutions with finite second moment are global but aggregate as $t \rightarrow \infty$ \cite{Blanchet08}.
This is not the case in $d \geq 3$, where critical solutions are uniformly bounded and cannot concentrate past a certain amount \cite{Yao11}. 
In \cite{Yao11}, a symmetrization inequality from \cite{KimYao11} (see also \cite{DiazNagai95}), allows one to extend the result for radially symmetric solutions to all solutions. 
Currently no such symmetrization inequality is available in the inhomogeneous case, so we are restricted still to the radially symmetric result. 
The result is proved using a mass supersolution, however, it is important to note that since radial monotonicity is not generally conserved by \eqref{def:ADD}, unlike the homogeneous case, a mass supersolution is insufficient to deduce Theorem \ref{thm:GE_CritMass}, hence the necessity of also using Theorem \ref{thm:Concentration} (see \S\ref{sec:GECritMass} for further discussion).  

\begin{theorem}[Global Existence at Critical Mass] \label{thm:GE_CritMass}
Let $d\geq 3$ and $a(x)\in C^\infty(\mathbb{R}^d)$ be strictly positive and $\gamma(x) \in C^\infty(\Real^d)$ be non-negative such that $D^\alpha a$ and $D^\alpha \gamma$ are bounded for all multi-indices $\alpha$.
Suppose additionally that $a(x)$ and $\gamma(x)$ are radially symmetric. 
Then any radially symmetric weak solution $u(t)$ to \eqref{def:ADD} with $M(u) = M_c$ exists globally and $u(t) \in L^\infty(\Real^+ \times \Real^d)$. 
\end{theorem}

\medskip

The following summarizes our finite time blow-up results, which demonstrates that under certain hypotheses, 
Theorems \ref{thm:GE} and \ref{thm:GE_CritMass} are sharp.

\begin{theorem}[Finite Time Blow-Up] \label{thm:blowup}
Let $\gamma(x)\equiv  0$ and let $a \in C^1(\mathbb{R}^d)$ be radially symmetric, strictly positive and such that both $a$ and $\grad a$ are bounded.
Suppose also that $a(0) = \min a(x)$ and that there exists a neighborhood $\abs{x} < \delta_0$ such that $a(x)$ is radially non-decreasing. 
Then for all $M > M_c$, there exists a solution $u(t)$ with $M(u) = M$ which blows up in finite time, e.g. $T_+(u(0)) < \infty$. 
\end{theorem}

\begin{remark}
The requirement that $a$ and $\grad a$ be uniformly bounded are only used to satisfy the hypotheses of Theorem \ref{thm:loc_theory}, which ensures we have a well-understood local existence, uniqueness and stability theory.   
\end{remark}

Naturally, blow-up solutions constructed in Theorem \ref{thm:blowup} are required to concentrate a sufficient amount of mass near where that minimum is achieved. 
Exactly how concentrated the initial data is required to be is characterized by \eqref{ineq:massorderedic} in the following proposition (which implies Theorem \ref{thm:blowup}), which requires at least part of the initial data be more concentrated than a particular rescaled extremal of the sharp HLS (Theorem \ref{thm:SharpHLS}).
Theorem \ref{thm:blowup} is proved by comparing the true solution against a barrier, and \eqref{ineq:massorderedic} below is the requirement that the solution and the barrier are ordered at time zero.  
Remark \ref{rmk:Concentration} clarifies how \eqref{ineq:massorderedic} requires $u_0$ to concentrate around where the minimum is achieved at least when $M(u_0) \searrow M_c$.  

\begin{proposition} \label{prop:InitialMassConcen}
Assume the hypotheses of Theorem \ref{thm:blowup} hold. Let $u_0 \in L^1_+(\Real^d; (1+\abs{x}^2)dx) \cap C^0(\Real^d)$ be radially symmetric such that $M_c <M(u_0)$ and suppose that there is an $R_0 \leq \delta_0$ and an $M_0$ with $M_c < M_0 < M(u_0)$ such that for $0 \leq r \leq R_0$ we have
\begin{equation}
\int_{\abs{x} \leq r} \left(\frac{a(R_0)^{1/2}}{R_0}\right)^d V\left(\frac{x}{R_0}\right) dx \leq \left(\frac{M_c}{M_0}\right)\int_{\abs{x} \leq r} u_0(x) dx. \label{ineq:massorderedic}
\end{equation}
Then the weak solution $u(t)$ associated with $u_0$ blows up in finite time. 
Moreover, if we define $\mu := M_c/M_0 < 1$, we have the following estimate of the blow-up time: 
\begin{equation}
T_+(u_0) \leq \mu^{2/d-1}\frac{\sigma R_0^d}{a(R_0)dM_c(\mu^{-2/d}-1)} < \infty, \label{ineq:Tplus} 
\end{equation}
where $\sigma$ is the surface area of the unit sphere in $\Real^d$. 
\end{proposition} 
\begin{remark} \label{rmk:Concentration}
In light of Remark \ref{rmk:RescaledV}, \eqref{ineq:massorderedic} implies that 
\begin{equation*}
M_0\left( \frac{a(R_0)}{a(0)} \right)^{d/2} \leq \int_{\abs{x} \leq R_0} u_0(x) dx. 
\end{equation*}
Hence, in order to construct blow-up solutions with $M(u_0) \searrow M_c$, in general we need to choose increasingly concentrated initial data by sending $R_0 \rightarrow 0$.
\end{remark}

\begin{remark} \label{rmk:Mobility} 
Arguably, the most natural inhomogeneity in the PDE for $u(t,x)$ is a mobility (which weights the cost of mass transport) that depends on space, that is $b(x) > 0$ such that 
\begin{equation} \label{def:ADDInh2} 
\left\{
\begin{array}{l}
  u_t = \grad \cdot (b(x)u \grad \frac{m}{m-1}u^{m-1} - \grad c) \\ 
  -\grad \cdot (a(x)\grad c) + \gamma(x)c = u. 
\end{array}
\right.
\end{equation}
In this case the free energy is still given by \eqref{def:F}; the gradient flow structure is only changed via the metric. The Wasserstein metric would be replaced by a generalized version (still based on $L^2$) that accounts for the inhomogeneity in the cost of mass transport. 
As Theorems \ref{thm:GE} and \ref{thm:Concentration} are proved using essentially only \eqref{def:F}, the results still hold for 
\eqref{def:ADDInh2} without any relevant change in statement or proof.
The result of Theorem \ref{thm:blowup} will also hold under suitable monotonicity hypotheses on $b(x)$ and the proof
is only a minor modification of the current one. 
Other possible variants of \eqref{def:ADD} could involve mobilities which are different in each term, however this
would probably change the variational structure and the methods we employ to prove Theorems \ref{thm:GE} and \ref{thm:Concentration} would no longer apply.
Regardless, this is a very different model from a physical standpoint, as this corresponds to adjusting the organism response to stimulus dependent on space.  
\end{remark}

As mentioned above, Theorem \ref{thm:blowup} and Proposition \ref{prop:InitialMassConcen} also provide new results for the homogeneous problem, $a(x) \equiv 1$, $\gamma(x) \equiv 0$. 
The virial method used in \cite{Blanchet09} proves blow-up for all solutions with negative free energy, without the need for radial symmetry. On the other hand, Theorem \ref{thm:blowup} and Proposition \ref{prop:InitialMassConcen} require radial symmetry but do not require any assumptions on the free energy. Indeed we have the following corollary of Proposition \ref{prop:InitialMassConcen}, which in particular, shows the existence of solutions with arbitrarily large initial free energy that blow up in finite time. 

\begin{theorem} \label{cor:HomogProblem}
Let $u_0 \in L^1_+(\Real^d;(1+\abs{x}^2)dx) \cap C^0(\Real^d)$ be radially symmetric with $u_0(0) > 0$ and satisfy
\begin{equation*}
M(u_0) > \left( \frac{2}{(m-1)C_\star c_d} \right)^{d/2}. 
\end{equation*}
Then the weak solution $u(t)$ associated with $u_0$ of the scale-invariant problem $(a(x) \equiv 1$, $\gamma(x) \equiv 0)$ blows up in finite time. 
In particular, for every $F_0 \geq 0$, there exists a solution $u(t)$ with $\F(u(0)) > F_0$ which blows up in finite time. 
\end{theorem}

\begin{remark} 
It would be interesting to determine if there exist solutions which have positive energy at blow-up time. 
Such a result would require deducing more precise information about the blow-up structure.
\end{remark}

\section{Global Existence} \label{sec:GE}
The proof of Theorems \ref{thm:GE} and \ref{thm:Concentration} hinge primarily on providing a precise decomposition of the potential energy,
\begin{equation*}
\int u(x) c(x) dx, 
\end{equation*}
into a leading order critical part and another part that is subcritical. 
The use of pseudo-differential operators for this purpose is the subject of the following section. 

\subsection{Approximate Inverse of Chemo-attractant PDE} \label{sec:ApproxInverse}
We use the standard symbol classes studied in for example \cite{BigStein}, summarized in the following definition. 
\begin{definition}[Symbol Class $S^{s}$, $s \in \Real$] \label{def:SymbolClass}
Suppose $b(x,\xi) \in C^\infty(\Real_x^d \times \Real_\xi^d)$ satisfies 
\begin{equation*}
\abs{\partial_x^\beta \partial_\xi^{\alpha}b(x,\xi)} \lesssim_{\beta,\alpha} (1 + \abs{\xi})^{s - \abs{\alpha}}, 
\end{equation*}
for multi-indices $\alpha,\beta$. 
Then we say both $b$ and the associated pseudo-differential operator ($\Psi$DO) $T_b$ defined by
\begin{equation*}
T_bf(x) = \frac{1}{(2\pi)^{d/2}}\int b(x,\xi)\hat{f}(\xi)e^{ix \xi} d\xi 
\end{equation*}  
are in the symbol class $S^{s}$ and say the symbol $b$ or the operator $T_b$ are of order $s$. 
We also denote $b(x,\xi) = \textup{sym}(T_b)$.
\end{definition}
Notice that with this definition $S^{s_1} \subset S^{s_2}$ whenever $s_1 < s_2$. 
Also, since the symbols are required to be smooth, these operators do not carry too much low-frequency information, unlike multiplier or symbol classes that allow singularities at the origin.  
For the standard relevant facts regarding these symbol classes, such as the symbolic calculus, localization estimates, boundedness on Sobolev spaces and singular integral representations, see Chapter 6 of \cite{BigStein}.  

\vspace{10pt}

Consider the PDE
\begin{equation}
Lc := -a(x)\Delta c - \grad a(x)\cdot \grad c + \gamma(x)c = u. \label{def:Lc}   
\end{equation}
By definition, $L$ is a pseudo-differential operator in $S^2$:
\begin{equation*}
Lc = \frac{1}{(2\pi)^{d/2}}\int \left(\gamma(x) + a(x)\abs{\xi}^2 - i\xi \cdot \grad a(x) \right)\hat{c}(\xi)e^{i x\xi} d\xi.  
\end{equation*}
Consider the approximate inverse of $L$, the $S^{-2}$ class $\Psi$DO 
\begin{equation*}
A_Hu := \frac{1}{(2\pi)^{d/2}}\int \frac{\Phi(\xi)\hat{u}(\xi)e^{ix\xi}}{a(x)\abs{\xi}^2 - i\xi \cdot \grad a(x) + \gamma(x)} d\xi, 
\end{equation*}
where $\Phi(\xi) = \prod_{j=1}^d\phi(\xi_j)$ with $\phi(t)$ a smooth function such that $0 \leq \phi(t) \leq 1$ which is identically one for $\abs{t} \geq 1$ and vanishes in a neighborhood of zero. We remark that if $\gamma(x)$ is strictly positive, we do not need the cut-off $\Phi(\xi)$.  
By the symbolic calculus [Chapter 6, Theorem 2 \cite{BigStein}],  
\begin{equation*}
A_HLc = c + T_{E}c, 
\end{equation*}
where the operator $T_E \in S^{-1}$ and the associated symbol $E$ has the following asymptotic expansion
\begin{equation*}
E \sim \Phi(\xi) - 1 + \sum_{\abs{\alpha} \geq 1} \frac{(2\pi i)^{-\abs{\alpha}}}{\alpha !}\partial_\xi^\alpha \textup{sym}(A_H)\partial_x^\alpha \textup{sym}(L), 
\end{equation*}
in the sense that the error in truncating the series for $N>\abs{\alpha}$ is a symbol of class $S^{-1-N}$. 
   
   \vspace{10pt}
   
Although both $A_H$ and $T_E$ are bounded operators from $L^p$ to itself for $1 < p < \infty$ (Chapter 6 \cite{BigStein}), if $\gamma(x)$ is not strictly positive then $L^{-1}$, the true inverse, is not. 
Necessarily, the low-frequency portion of $L^{-1}$ is still present implicitly in $T_Ec$. 
Instead of being bounded on $L^p$ to itself, $L^{-1}$ satisfies the following: for $1 < q < d/2$, $d/(d-2) < p < \infty$ and $\frac{2}{d} + \frac{1}{p} = \frac{1}{q}$, 
\begin{equation}
\norm{c}_p = \norm{L^{-1}u}_p \lesssim \norm{u}_q. \label{ineq:HomogenCbound}
\end{equation}
This can be seen at least formally by multiplying both sides of \eqref{def:Lc} by $c^{(d-2)p/d - 1}$ integrating, and applying the homogeneous Sobolev embedding (in the homogeneous case, \eqref{ineq:HomogenCbound} is the Hardy-Littlewood-Sobolev inequality). 

\vspace{10pt}

We may formally write down the operator $A_H$ as a singular integral operator by interchanging the integrals: 
\begin{align*}
A_H u(x) & = \frac{1}{(2\pi)^{d}} \int\int u(y)\frac{\Phi(\xi)e^{i\xi(x-y)}}{a(x) \abs{\xi}^2 - i\xi \cdot \grad a(x) + \gamma(x)} dy d\xi \\  
&= \frac{1}{(2\pi)^{d}} \int u(y) \left[\int\frac{\Phi(\xi)e^{i\xi(x-y)}}{a(x) \abs{\xi}^2 - i\xi \cdot \grad a(x) + \gamma(x)} d\xi \right] dy \\
& := \int u(y) K_H(x,y) dy.  
\end{align*}
The integral for $K_H(x,y)$ is not absolutely convergent so we cannot naively apply Fubini's theorem in the above computation rigorously, but it can be justified by a standard limiting procedure, as in \cite{BigStein}.
The key technical lemma for the proof of Theorem \ref{thm:GE} is the following characterization of $K_H(x,y)$. 

\begin{lemma}[Asymptotic Expansion for $K_H(x,y)$] \label{lem:AsympExpan} 
Let $K_H(x,y)$ be defined as above by the conditionally convergent integral 
\begin{equation*}
K_H(x,y) = \frac{1}{(2\pi)^d} \int\frac{\Phi(\xi)e^{i\xi(x-y)}}{a(x) \abs{\xi}^2 - i\xi \cdot \grad a(x) + \gamma(x)} d\xi. 
\end{equation*} 
Then we then have the following asymptotic expansion which holds uniformly in $x \in \Real^d$, 
\begin{align}
K_H(x,y) %& = \frac{\Gamma(d/2 + 1)}{d(d-2)\pi^{d/2}a(x)} \abs{x-y}^{2-d} + o(\abs{x-y}^{2-d}) \;\;\; \hbox{ as } y \rightarrow x \label{eq:AsymptK} \\
         & = \frac{c_d}{a(x)}\abs{x-y}^{2-d}  + o(\abs{x-y}^{2-d}) \;\;\; \hbox{ as } y \rightarrow x, \label{eq:AsymptK}
\end{align}
with $c_d$ given above by \eqref{def:cd}.
Moreover, recall that for all $\delta > 0$ and $N > 0$ (see for example pg 235 \cite{BigStein}),  
\begin{equation}
\abs{K_H(x,y)} \lesssim_{\delta,N} \abs{x-y}^{-N}, \;\;\; \abs{x-y} > \delta. \label{ineq:LinftyBoundK}
\end{equation}
\end{lemma}

\begin{proof}
The bound \eqref{ineq:LinftyBoundK} is a standard consequence of $A_H \in S^{-2}$. 
Such localization should not be surprising since the the low frequency contribution of $L^{-1}$ is not included in $A_H$ due to the cut-off $\Phi$. 
Hence, we focus on \eqref{eq:AsymptK}.
Note the trick 
\begin{equation*}
\frac{1}{D} = \int_0^\infty e^{-tD} dt. 
\end{equation*} 
Hence, 
\begin{align*}
K_H(x,y) & = \frac{1}{(2\pi)^{d}}\int_0^\infty \int \Phi(\xi) e^{i\xi\cdot(x-y + t\grad a(x)) - t a(x)\abs{\xi}^2 - t\gamma(x)} d\xi dt \\
& = \frac{1}{(2\pi)^d}\prod_{j = 1}^{d}\int_0^\infty e^{-t\gamma(x)} \int_{-\infty}^\infty \phi(\xi_j) e^{i\xi_j(x_j-y_j + t\partial_{x_j} a(x)) - t a(x)\xi_j^2} d\xi_j dt. 
\end{align*}
Now define the complex change of variable $z_j = (ta(x))^{1/2} \xi_j - i\frac{x_j - y_j + t\partial_{x_j}a(x)}{2(ta(x))^{1/2}}$, 
\begin{align*}
\int_{-\infty}^\infty \phi(\xi_j) e^{i\xi_j(x_j-y_j + t\partial_{x_j} a(x)) - t a(x)\xi_j^2} d\xi_j dt & = \frac{e^{-\frac{\abs{x_j - y_j + t\partial_{x_j}a(x)}^2}{4ta(x)}}}{(ta)^{1/2}}\int_{\Gamma_j} \phi\left(\frac{\RealPart z_j}{(ta(x))^{1/2}}\right)e^{-z_j^2} dz_j\\ & :=  \frac{e^{-\frac{\abs{x_j - y_j + t\partial_{x_j}a(x)}^2}{4ta(x)}}}{(ta(x))^{1/2}}f_j(t), 
\end{align*}
where $\Gamma_j$ is the contour $\set{\ImPart z_j = \frac{x_j - y_j + t\partial_{x_j}a(x)}{2(ta(x))^{1/2}}}$.

\vspace{10pt}

Applying the above change of variables to the expression for $K_H(x,y)$ implies
\begin{align*}
K_H(x,y) = \frac{e^{-(x-y)\grad a(x)/(2a(x))}}{(2\pi)^d (a(x))^{d/2}} \int_0^\infty t^{-d/2}\left[\prod_{j=1}^df_j(t) \right] e^{-t\gamma(x) -t\frac{\abs{\grad a(x)}^2}{4a(x)}-\frac{\abs{x-y}^2}{4ta(x)}} dt.  
\end{align*}
We make the following additional change of variables,
\begin{equation*}
t = \frac{\abs{x-y}^{2}}{4a(x)\zeta^2}, 
\end{equation*}
which then yields
\begin{equation*} 
\zeta = \frac{\abs{x-y}}{(4ta(x))^{1/2}} \quad \hbox{ and } \quad dt = -\frac{\abs{x-y}^{2}}{2a(x) \zeta^3} d\zeta.
\end{equation*}
In terms of $\zeta$, $K_H(x,y)$ can be now written as
\begin{align*}
K_H(x,y) = \frac{e^{-(x-y)\grad a(x)/(2a(x))}}{(\pi)^d(2a(x))} \abs{x-y}^{2-d} \int_0^\infty \zeta^{d-3}\left[\prod_{j=1}^df_j(t(\zeta)) \right] e^{-\frac{\abs{x-y}^2}{4a(x)\zeta^2}\left( \frac{\abs{\grad a(x)}^2}{4a(x)} + \gamma(x)\right)-\zeta^{2}} d\zeta.   
\end{align*}
Due to the smoothness of $\phi$ and $a$ and the strict lower bound on $a$ we have the uniform (in $x$) convergence of the integral (note that here $\zeta$ is fixed and $z_j$ is the complex integration variable)
\begin{align*}
\lim_{y \rightarrow x}f_j(t(\zeta)) = \lim_{y \rightarrow x}\int_{\Gamma_j} \phi\left( \frac{2 \abs{\zeta} \RealPart z_j}{\abs{x-y}} \right) e^{-z_j^2} dz_j = \int_{\Real} e^{-z_j^2} dz_j = \pi^{1/2}.  
\end{align*}
Similarly we also have
\begin{equation*}
\lim_{y \rightarrow x} \int_0^\infty \zeta^{d-3}\left[\prod_{j=1}^df_j(t(\zeta)) \right] e^{-\frac{\abs{x-y}^2}{4a(x)\zeta^2}\left( \frac{\abs{\grad a(x)}^2}{4a(x)} + \gamma(x)\right)-\zeta^{2}} d\zeta = \pi^{d/2} \int_0^\infty \zeta^{d-3}e^{-\zeta^2} d\zeta, 
\end{equation*}
uniformly in $x \in \Real^d$ due to the uniform continuity of $\grad a(x), a(x),$ and $\gamma(x)$ as well as the strict positivity of $a$.
Recalling elementary facts about the Gamma function,
\begin{equation*}
\int_0^\infty \zeta^{d-3}e^{-\zeta^2} d\zeta = \frac{1}{2}\Gamma\left(\frac{d}{2}-1\right) = \frac{2}{d(d-2)}\Gamma\left(\frac{d}{2} + 1 \right). 
\end{equation*}
Hence, 
\begin{align*}
K_H(x,y) & = \frac{\Gamma\left(\frac{d}{2}+1\right)}{d(d-2)\pi^{d/2}a(x)}\abs{x-y}^{2-d} + o_{y \rightarrow x}\left(\abs{x-y}^{2-d}\right),
\end{align*}
and \eqref{eq:AsymptK} is proved. 
\end{proof} 

The error term $T_Ec$ in the approximate inverse can be controlled by the following lemma. Here we take advantage of the smoothing nature of $T_E \in S^{-1}$ to show that in the potential energy, this term is subcritical in the sense that the effective power of $\norm{u}_m$ associated with the term is strictly less than $m$.   
Naturally, one must eventually use \eqref{ineq:HomogenCbound} in order to prove this lemma.  

\begin{lemma}[The error term is subcritical] \label{lem:subcriticalT_E} 
Let $d \geq 3$, suppose $u \in L^1 (\mathbb{R}^d) \cap L^{m}(\mathbb{R}^d)$ and $c$ is the strong solution of \eqref{def:Lc} which vanishes at infinity. Then, 
\begin{equation}
\int u T_E c dx \lesssim \norm{u}_1^{2-2\theta}\norm{u}_{m}^{2\theta}, \label{ineq:TEc}
\end{equation}
for some $0 < 2\theta < m=2-2/d$. 
\end{lemma}
\begin{proof} 
Define 
\begin{equation*}
p = \frac{4d}{2d-3}. 
\end{equation*}
Define the standard Fourier multiplier $\widehat{\jap{\grad}^{s}f} := \left(1 + \abs{\xi}^2\right)^{s/2}\hat{f}(\xi)$. 
Since the multiplier $\jap{\grad}^{1/2}$ is self-adjoint we have, 
\begin{align*}
\abs{\int uT_Ec dx} \leq \norm{\jap{\grad}^{-1/2}u}_{\frac{p}{p-1}}\norm{\jap{\grad}^{1/2}T_Ec}_p
\end{align*}
Define 
\begin{equation*}
\frac{1}{q} = \frac{1}{p} + \frac{2}{d} = \frac{2d + 5}{4d} < 1. 
\end{equation*}
Definition \ref{def:SymbolClass} implies that since $T_E \in S^{-1}$, we also have $T_E\in S^{-1/2}$. 
Hence, $T_E$ is a bounded operator $L^{p}$ to $W^{1/2,p}$ [Chapter 6, Proposition 5 \cite{BigStein}]. 
Using this and the elliptic $L^p$ estimate \eqref{ineq:HomogenCbound} we have
\begin{equation*}
\norm{\jap{\grad}^{1/2}T_Ec}_p \lesssim \norm{c}_p \lesssim \norm{u}_q. 
\end{equation*}
One may easily verify that
\begin{equation*}
\frac{1}{q} = \frac{p-1}{p} + \frac{1}{2d}, 
\end{equation*}
and therefore the inhomogeneous Sobolev embedding theorem implies
\begin{equation*}
\norm{\jap{\grad}^{-1/2}u}_{\frac{p}{p-1}} \lesssim \norm{u}_q.  
\end{equation*}
Hence, we see the relevance of $q$ as we have in total,
\begin{align*}
\abs{\int u T_Ec dx } \lesssim \norm{u}_q^2. 
\end{align*}
In order to interpolate between $L^1$ and $L^m$, we need $q < m=2-2/d$, which follows easily from $d \geq 3$. 
Then, for 
\begin{equation*}
\theta = \frac{(2d-5)(2d-2)}{4d(d-2)} \in (0,1)
\end{equation*}
we have, 
\begin{equation*}
\abs{\int uT_Ec dx} \lesssim \norm{u}_1^{2-2\theta}\norm{u}^{2\theta}_{m}.
\end{equation*}
To prove subcriticality it remains to confirm that we have $2\theta < m$, which again follows from $d \geq 3$.   
\end{proof}

\subsection{Proof of Theorem \ref{thm:GE}}
In this section we complete the proof of Theorem \ref{thm:GE}. 

\begin{proof} 
We prove Theorem \ref{thm:GE} by producing a uniform in time bound on the entropy (which is basically just the $L^m$ norm).  This in turn proves that the solution is uniformly equi-integrable and Theorem \ref{thm:loc_theory} completes the proof.

By the energy dissipation inequality \eqref{ineq:EnrDiss} and the definition of $A_H$,
\begin{align*}
\frac{1}{m-1}\int u^m dx - \frac{1}{2}\int uA_H u dx + \frac{1}{2}\int uT_E c dx \leq \F(u_0). 
\end{align*}
By \eqref{ineq:TEc} (Lemma \ref{lem:subcriticalT_E}) we then have, for some $0 < 2\theta < m = 2-2/d$, 
\begin{align*}
\frac{1}{m-1}\int u^m dx - \frac{1}{2}\int uA_H u dx - C\norm{u}_1^{2-2\theta}\norm{u}_{m}^{2\theta} \leq \F(u_0). 
\end{align*}
Using Lemma \ref{lem:AsympExpan}, for every $\epsilon > 0$, we may choose a $\delta > 0$ such that
\begin{align*}
\int uA_H u dx \leq \int\int_{\abs{x-y}<\delta} \left(\frac{c_d}{a(x)} + \epsilon\right)\frac{u(x)u(y)}{\abs{x-y}^{d-2}} dx dy + C(\delta)\norm{u}_1^2,
\end{align*}
where $c_d$, given by \eqref{def:cd}, is the normalization constant in the Newtonian potential. 
Hence, by the sharp Hardy-Littlewood-Sobolev inequality \eqref{ineq:NHLS} we have,  
\begin{align*}
\left(\frac{1}{m-1} - \frac{C_\star}{2}\left(\frac{c_d}{\min a(x)} + \epsilon\right)\norm{u}_1^{2/d}\right)\norm{u}_{m}^{m}\leq C\norm{u}_1^{2-2\theta}\norm{u}_{m}^{2\theta} + C(\delta)\norm{u}_1^2 + \F(u_0). 
\end{align*}
If $M(u) < M_c$ with $M_c$ given by \eqref{def:critmass} then we may choose $\epsilon$ sufficiently small such that the first term is positive. 
Since $m > 2\theta$ this then implies a global uniform-in-time bound on $\norm{u(t)}_{m}$. This in turn implies a global $L^\infty$ bound on $u(t)$ by the continuation criterion in Theorem \ref{thm:loc_theory}. 
\end{proof} 
\begin{remark}\label{rmk:R2inGE}
It appears the proof of Theorem \ref{thm:GE} would require some non-trivial refinement in order to treat the $\Real^2$ case. 
Naturally, the asymptotic expansion of $K_H(x,y)$ along the diagonal $x \sim y$ would need to be refined in order to capture the logarithmic singularity accurately (Lemma \ref{lem:AsympExpan}). 
A likely more delicate task would be adjusting Lemma \ref{lem:subcriticalT_E} to yield an error which is subcritical relative to the positive part of the entropy $\int u (\log u)_+ dx$.
\end{remark}

\subsection{Proof of Theorem \ref{thm:Concentration}}
Like Theorem \ref{thm:GE}, the proof of Theorem \ref{thm:Concentration} is essentially a precise estimate of the potential energy which allows the free energy to uniformly control the $L^m$ norm of the solution.  
In Theorem \ref{thm:GE}, after the results of \S\ref{sec:ApproxInverse}, the assumption of subcritical total mass was used to make this estimate. 
Theorem \ref{thm:Concentration} refines this result to show that the free energy uniformly controls the $L^m$ norm unless a critical amount of mass concentrates into a single point or at least along a subsequence of points escaping to infinity.  
The main idea of the proof is to combine the related result for critical mass solutions in 2D found in \cite{Blanchet08} with a more intricate concentration compactness-style decomposition. 
It is a little technical, but as with many concentration compactness arguments, the key idea is simple: if the solution does not concentrate a critical amount of mass at blow-up, then we may divide $\Real^d$ into a finite number of domains, each containing subcritical mass at blow-up. In each region we may apply reasoning similar to Theorem \ref{thm:GE} to control the corresponding contribution to the free energy, noting that long-range interactions between different regions are ``sub-critical" since the leading order term in the approximation of the potential energy deduced in \S\ref{sec:ApproxInverse} will not be present.  

\begin{proof}
We suppose Theorem \ref{thm:Concentration} is false and go towards a contradiction. 
For the duration of the proof, the notation $B(x,r)$ will denote the \emph{closed} ball with center $x$ and radius $r$. 
Moreover, let $\bar{a} = \min a$, $A = \max a$ and define 
\begin{equation*}
M^\star_c := \left( \frac{2}{(m-1)c_d C_\star}\right)^{d/2}. 
\end{equation*}

The assumption that Theorem \ref{thm:Concentration} is false is equivalent to the existence of a time sequence $t_k \to T_+$ and $\epsilon > 0$ such that for all sequences $\set{x_k}_{k = 1}^\infty \subset \Real^d$, 
\begin{equation}
\limsup_{r \rightarrow 0} \limsup_{k \rightarrow \infty} \left(\int_{B(x_k,r)}u(t_k,x) dx - a(x_k)^{d/2}M_c^\star \right) < -2\epsilon. \label{def:epsilonnstar}
\end{equation}
Let $u_k(x):= u_k(t_k,x)$.
The purpose of the additional factor of two is purely cosmetic and will be apparent in what follows. 
Since the proof mostly rests on an estimate of the potential energy, define the following (since $L^{-1}$ is self-adjoint)  
\begin{equation*}
PE(f,g) := \frac{1}{2}\int f L^{-1}g dx = \frac{1}{2}\int g L^{-1}f dx. 
\end{equation*}
The proof begins with the following equivalent version of the concentration compactness lemma found in \cite{LionsCC84}
\begin{lemma} \label{lem:CC} 
Let $\set{u_k}$ be a sequence of non-negative functions in $L^1(\Real^d)$ with $\norm{u_k}_1 = M \in (0,\infty)$. 
Then for all $\delta > 0$, there exists a subsequence (not relabeled) of $\set{u_k}$ such that the following holds: there exists an $L < \infty$ such that there exists sequences $x_k^l \subset \Real^d, R^l \in (0,\infty)$ and $u_k^l,e_k,u_k^V \in L^1$ (which all depend on $\delta$) for all $l \in \set{1,...,L}$ which satisfy
\begin{equation}
u_k = \sum_{l = 1}^L u_k^l + u_k^V + e_k, 
\end{equation} 
with the following properties: 
\begin{itemize}
\item[(i)] The supports of $u_k^l$, $e_k$ and $u_k^V$ are all disjoint.    
\item[(ii)] $\supp u_k^l \subset B(x_k^l,R^l)$ and $\lim_{k \rightarrow \infty} \abs{x_k^{m} - x_k^{l}} = \infty$ if $l \neq m$.  
\item[(iii)] $\limsup_{k \rightarrow \infty}\norm{e_k}_1 < \delta$. 
\item[(iv)] For all $l$, $\lim_{k \rightarrow \infty}\textup{dist}\,(\supp u_k^l,\supp u_k^V) = \infty $. Moreover, 
$u_k^V$ is vanishing in the sense that for all $R > 0$,  
\begin{equation*}
\lim_{k \rightarrow \infty} \sup_{x \in \Real^d} \int_{B(x,R)} u_k^V(y) dy = 0.  
\end{equation*}  
\end{itemize}
\end{lemma}
\begin{proof} 
The proof proceeds by iterating Lemma II.1 in \cite{LionsCC84}. 
As the starting point of the algorithm, define $\rho_k^1 = u_k$. 
Now consider the inductive step as follows. Suppose that $u_k^l$ which satisfy the conclusions of Lemma \ref{lem:CC} have been determined for $0 \leq l < j$ and suppose that $v_k^l$ are a sequence which satisfies the properties: 
\begin{itemize}
\item[(a)] $\limsup_{k\rightarrow \infty}\norm{v_k^l}_{1} < 2^{-l-1}\delta$
\item[(b)] For all $0 \leq l < j$, the support of $v_k^l$ are disjoint from $u_k^m$ for all $m$ and $v_k^m$ for all $m \neq l$.
\end{itemize}
We are also given $\rho_k^j := u_k - \sum_{0 \leq l < j} u_k^l + v_k^l$. 
By Lemma II.1 in \cite{LionsCC84} one of three possibilities occurs up to extraction of a subsequence (not relabeled).  
If $\rho_k^j$ is tight up to translation then there exists a sequence $x_k^j$ and a $R^j$ such that 
\begin{align*}
\int_{\Real^d \setminus B(x_k^j,R^j)} \rho_k^j(y) dy < 2^{-j}\delta.  
\end{align*}
Hence, define $u_k^j := \rho_k^j |_{B(x_k^j, R^j)}$ and $e_k := \rho_k^j|_{\Real^d \setminus B(x_k^j,R^j)} + \sum_{0 \leq l < j} v_k^l$. 
We may terminate the algorithm here and set $L = j$. 
If instead $\rho_k^j$ is vanishing then simply define $u_k^V := \rho_k^j$ and $e_k = \sum_{0 \leq l < j} v_k^l$. Terminate the algorithm here and set $L= j-1$. 
In the case of dichotomy, decompose $\rho_k^j$ as in Lemma II.1 of \cite{LionsCC84} as
\begin{align*}
\rho_k^j = \rho_k^{j,1} + \rho_k^{j,2} + v_k^{j},  
\end{align*}
with $\limsup_{k \rightarrow \infty} \norm{v_k^1}_1 < 2^{-j-1}\delta$. The following additional properties are satisfied: there exists a sequence $x_k^j$ and an $R^j > 0$ such that $\rho_k^{j,1} = \rho_k^j|_{B(x^j_k,R^j)}$ and a sequence $R_k > 0$ with $\lim_{k \rightarrow \infty} R_k = \infty$ such that $\rho_k^{j,2} = \rho^j_k|_{\Real^d \setminus B(x_k^j,R_k)}$.
This latter fact is what ensures (ii) and the  first part of (iv). 
Take $u_k^{j} = \rho_k^{j,1}$ (with the associated $x_k^j$ and $R^j$). 
If $\limsup_{k \rightarrow \infty}\norm{\rho_k^{j,2}}_1 < \delta/2$ then terminate the algorithm here with $L = j$ and define 
$e_k = \rho_k^{j,2} + \sum_{0 \leq l \leq j}v_k^{l}$.  
If not, continue the algorithm with $\rho_k^{j+1} = \rho_k^{j,2}$.  
This termination condition ensures that the algorithm terminates in at most $\mathcal{O}(\delta^{-1})$ steps.  
\end{proof} 

We will fix $\delta = \delta(a,\e)$ later depending only on $\epsilon$, $a$ and the total mass $M$. 
We apply Lemma \ref{lem:CC} to $u_k := u(t_k)$, $t_k \nearrow T_+$ with this $\delta$.  
This supplies the first step in the decomposition, which must be refined further. 
The following geometric decomposition lemma refines the decomposition by subdividing the $u_k^l$ into pieces of subcritical mass.
 The authors would like to acknowledge Jonas Azzam for assisting in the proof, which is based on an iterative application of the Besicovitch covering theorem \cite{Mattila}. 
We postpone the proof until after the proof of Theorem \ref{thm:Concentration}.
\begin{lemma} \label{lem:decompo} 
Let $u_k=u_k(t_k)$ as given above. Then for any $\delta>0$ there exists a collection of $N = N(\delta) < \infty$ disjoint closed balls with centers $\hat{x}_k^j \subset \Real^d$ and radii $r_j > 0$, $1 \leq j \leq N$ which satisfy the following properties for all $k$ sufficiently large,  
\begin{itemize}
\item[(1)] For all $j$, $\int_{B(\hat{x}_k^j,r_j)} u_k(y) dy < \left(\liminf_{k \rightarrow \infty} \min_{ \abs{x - \hat{x}_k^j} < r_j} a(x) \right)^{d/2}M_c^\star - \epsilon$. 
\item[(2)] $\int_{\Real^d \setminus \cup_{j = 1}^N B(\hat{x}_k^j,r_j)} u_k(y) - u_k^V dy < 2\delta$. 
\item[(3)] There exists a constant $c_0 > 0$ independent of $k$ such that $\min_{1\leq i,j \leq N}\abs{\hat{x}_k^j - \hat{x}_k^i} \geq c_0$.
\end{itemize}
\end{lemma}   
We proceed to prove Theorem Theorem \ref{thm:Concentration}: with the decomposition obtained in Lemma~\ref{lem:decompo}.  Let us denote $B_{k,j} = B(\hat{x}_k^j, r_j)$ and $E_k := \Real^d \setminus (\cup_{j = 1}^{N} B_{k,j})$. 
These balls allow us to decompose the potential energy as follows, 
\begin{align}
PE(u_k,u_k) = \sum_{1 \leq j,i \leq N} PE(u_k|_{B_{k,j}},u_k|_{B_{k,i}}) + \sum_{1 \leq i \leq N} PE(u_k|_{B_{k,i}^l}, u_k|_{E_k}) + PE(u_k|_{E_k}, u_k|_{E_k}). \label{eq:PEdecomp}
\end{align}
First consider the last term, using the pseudo-differential approximation, 
\begin{align}
PE(u_k|_{E_k}, u_k|_{E_k}) & = \int\int_{E_k\times E_k} \hspace{-25pt} u_k(x)u_k(y) K_H(x,y) dx dy - \int_{E_k} u_k(x) T_{E} L^{-1} (u_k|_{E_k})(x) dx.  \label{ineq:pek}
\end{align}
The error term is controlled by Lemma \ref{lem:subcriticalT_E}, 
\begin{equation*}
\int_{E_k} u_k(x) T_{E} L^{-1} (u_k|_{E_k})(x) dx \leq C_E\norm{u_k|_{E_k}}_m^{2\theta} \leq C_E\norm{u_k}_m^{2\theta},
\end{equation*}
where $C_E$ is the implicit constant in \eqref{ineq:TEc} times $\norm{u}_1^{2-2\theta}$.
Now we turn to the leading order term in \eqref{ineq:pek}. 
Using Lemma \ref{lem:AsympExpan}, for all $\tilde{\epsilon} > 0$, there is an $\eta > 0$ 
such that  
\begin{align*} 
\int\int_{E_k\times E_k} \hspace{-25pt} u_k(x)u_k(y) K_H(x,y) dx dy  & = \int\int_{E_k\times E_k \cap \set{\abs{x-y} < \eta}} \hspace{-25pt} u_k(x)u_k(y) K_H(x,y) dx dy \\ & \quad + \int\int_{E_k \times E_k \cap \set{\abs{x-y} \geq \eta}} \hspace{-25pt} u_k(x)u_k(y) K_H(x,y) dx dy \\ 
& \leq (c_d + \tilde{\epsilon})\int\int_{E_k \times E_k \cap \set{\abs{x-y} < \eta}}\frac{1}{a(x)} u_k(x)u_k(y)\abs{x-y}^{2-d} dx dy \\ & \quad + \sup_{x,y \in \Real^d} \sup_{\abs{x-y} \geq \eta}\abs{K_H(x,y)}\norm{u}^2_1.   
\end{align*}
Denote 
\begin{equation}\label{constant}
K(\tilde{\epsilon}) := \sup_{x,y \in \Real^d} \sup_{\abs{x-y} \geq \eta}\abs{K_H(x,y)}\norm{u}^2_1.
\end{equation} 
For this step it suffices to take $\tilde{\epsilon} = 1$. 
Focusing now on the first term, cover $\Real^d$ by cubes of unit volume, denoted $\set{Q_i}_{i = 1}^\infty = \mathcal{Q}$.  
By construction (Lemma \ref{lem:decompo} (2) and Lemma \ref{lem:CC} (iv)), there is a $k_0$ sufficiently large such that 
\begin{equation}
\sup_{k \geq k_0} \sup_{Q \in \mathcal{Q}} \int_{Q} u_k|_{E_k}  dy < 4\delta. \label{ineq:ekukVconcentration}
\end{equation}
Hence, by \eqref{ineq:Gconvo} and H\"older's inequality, 
\begin{align*}
\int\int_{E_k \times E_k \cap \set{\abs{x-y} < \eta}}\frac{1}{a(x)} u_k(x)u_k(y)\abs{x-y}^{2-d} dx dy & \\ & \hspace{-5cm} = \sum_{Q_i,Q_j \in \mathcal{Q}} \int\int_{Q_i \times Q_j \cap E_k \times E_k \cap \set{\abs{x-y} < \eta}}\frac{1}{a(x)} u_k(x)u_k(y)\abs{x-y}^{2-d} dx dy  \\ 
& \hspace{-5cm} \lesssim \sum_{Q_i,Q_j \in \mathcal{Q}} \norm{u_k|_{E_k \cap Q_i}}_1^{1/d}\norm{u_k|_{E_k \cap Q_j}}_1^{1/d} \norm{u_k|_{E_k \cap Q_i}}_m^{m/2}\norm{u_k|_{E_k \cap Q_j}}_m^{m/2} \\ 
& \hspace{-5cm}\leq K_1\delta^{2/d}\norm{u_k}_m^m,   
\end{align*}
for some constant $K_1 > 0$ which depends only on dimension and $\bar{a}$. 
Putting the previous estimates together we have shown that,  
\begin{align}
PE(u_k|_{E_k},u_k|_{E_k}) \leq K_1\delta^{2/d}\norm{u_k}_m^m + C_E\norm{u_k}_m^{2\theta} + K(1). \label{ineq:pekDone}
\end{align}
Let us now handle the second term in \eqref{eq:PEdecomp}. 
Using the pseudo-differential approximation, 
\begin{align}
PE(u_k|_{\cup_{j} B_{k,j}}, u_k|_{E_k}) & = \int\int_{ \cup_{j} B_{k,j} \times E_k} \hspace{-25pt} u_k(x)u_k(y) K_H(x,y) dx dy - \int_{\cup_{j} B_{k,j}} \hspace{-10pt} u_k(x) T_{E} L^{-1} (u_k|_{E_k})(x) dx.  \label{ineq:pekvBkim}
\end{align}
The error term is controlled by Lemma \ref{lem:subcriticalT_E}, 
\begin{align*}
\int_{\cup_{j} B_{k,j}} \hspace{-10pt} u_k(x) T_{E} L^{-1} (u_k|_{E_k})(x) dx & \leq C_E\norm{u_k|_{\cup_{j} B_{k,j}}}_m^{\theta}\norm{u_k|_{E_k}}_m^{\theta} \\
& \leq C_E\norm{u_k}_{m}^{2\theta}. 
\end{align*}
The leading order term in \eqref{ineq:pekvBkim} is estimated using Lemma \ref{lem:AsympExpan}, which implies for all $\tilde{\epsilon} > 0$, there is $\eta > 0$ such that  
\begin{align*} 
\int\int_{ \cup_{j} B_{k,j} \times E_k} \hspace{-25pt} u_k(x)u_k(y) K_H(x,y) dx dy  & = \int\int_{\cup_{j} B_{k,j} \times E_k \cap \set{\abs{x-y} < \eta}} \hspace{-25pt} u_k(x)u_k(y) K_H(x,y) dx dy \\ & \quad + \int\int_{\cup_{j} B_{k,j} \times E_k \cap \set{\abs{x-y} \geq \eta}} \hspace{-25pt} u_k(x)u_k(y) K_H(x,y) dx dy \\ 
& \leq (c_d + \tilde{\epsilon})\int\int_{\cup_{j} B_{k,j} \times E_k \cap \set{\abs{x-y} < \eta}}\frac{1}{a(x)} u_k(x)u_k(y)\abs{x-y}^{2-d} dx dy
 \\ & \quad + K(\tilde{\epsilon}), 
\end{align*}
where $K(\tilde\epsilon)$ is defined in \eqref{constant}. For this step it again suffices to take $\tilde{\epsilon} = 1$.  
By \eqref{ineq:Gconvo} and H\"older's inequality, 
\begin{align*}
\int\int_{\cup_{j} B_{k,j} \times E_k \cap \set{\abs{x-y} < \eta}}\frac{1}{a(x)} u_k(x)u_k(y)\abs{x-y}^{2-d} dx dy 
& \lesssim \norm{u_k|_{E_k}}_{1}^{1/d}\norm{u_k}_m^m. 
\end{align*}
Hence, the estimates put together imply, 
\begin{align} 
\sum_{j = 1}^N PE(u_k|_{B_{k,j}},u_k|_{E_k}) \leq K_2\delta^{1/d}\norm{u_k}_m^m + C_E\norm{u_k}_m^{2\theta} + K(1), \label{ineq:BkE}
\end{align}
for some fixed positive constant $K_2=K_2(d,M,\bar{a})$.  

Now we turn to the first term in \eqref{eq:PEdecomp}. 
Firstly, if $j \neq i$ then by arguments similar to above, using that the balls $B_{k,j}^l$ are disjoint for $k$ sufficiently large,    
\begin{align} 
PE(u_k|_{B_{k,j}},u_k|_{B_{k,i}}) \leq C_E\norm{u_k}_m^{2\theta} + K_3, \label{ineq:BkBj}
\end{align}
for some $K_3 > 0$ which depends on $M$, $d$, $a$ and the minimal distance between $B_{k,j}$ and $B_{k,i}$ which is bounded below by Lemma \ref{lem:decompo} (3). 
Hence, consider the case $i = j$ and for notational simplicity refer to $B_k := B_{k,j}$.  
Estimate now as above, 
\begin{align*}
PE(u_k|_{B_k},u_k|_{B_k}) & = \int\int_{B_k \times B_k} u_k(x) u_k(y) K_H(x,y) dx dy - \int_{B_k} u_k(x) T_E L^{-1}(u_k|_{B_k}) dx. 
\end{align*}
By Lemma \ref{lem:subcriticalT_E}, 
\begin{align*}
\int_{B_k} u_k(x) T_E L^{-1}(u_k|_{B_k}) dx \leq C_E\norm{u_k}_m^{2\theta}.  
\end{align*}
By Lemma \ref{lem:AsympExpan}, for all $\tilde{\epsilon} > 0$ there exists an $\eta = \eta(\tilde{\epsilon}) > 0$ such that 
\begin{align*}
\int\int_{B_k \times B_k} \hspace{-25pt} u_k(x)u_k(y) K_H(x,y) dx dy & = \int\int_{(B_k \times B_k) \cap \set{\abs{x-y} < \eta}} \hspace{-25pt} u_k(x)u_k(y) K_H(x,y) dx dy \\ & \hspace{8pt} + \int\int_{(B_k \times B_k) \cap \set{\abs{x-y} \geq \eta}} \hspace{-25pt} n_k(x)n_k(y) K_H(x,y) dx dy \\ 
& \leq (c_d + \tilde{\epsilon})\int\int_{(B_k \times B_k) \cap \set{\abs{x-y} < \eta}}\frac{1}{a(x)} u_k(x)u_k(y)\abs{x-y}^{2-d} dx dy \\ & \hspace{8pt} + K(\tilde{\epsilon}). 
\end{align*}  
Write
\begin{equation*} 
\bar{a}_{j} := \liminf_{k \rightarrow \infty} \min_{x \in B_k} a(x). 
\end{equation*}
By the sharp HLS, Theorem \ref{thm:SharpHLS}, 
\begin{align*}
\int\int_{(B_k \times B_k) \cap \set{\abs{x-y} < \eta}}\frac{1}{a(x)} u_k(x)u_k(y)\abs{x-y}^{2-d} dx dy & \leq \frac{C_\star}{\bar{a}_{j}}\norm{u_k|_{B_k}}^{2/d}\norm{u_k|_{B_k}}_m^m. 
\end{align*}
Hence, 
\begin{align}
PE(u_k|_{B_k},u_k|_{B_k}) \leq \frac{C_\star}{2 \bar{a}_{j}}(c_d + \tilde{\epsilon})\norm{u_k|_{B_k}}_1^{2/d}\norm{u_k|_{B_k}}_m^m + C_E\norm{u_k}_m^{2\theta} + K(\tilde{\epsilon}). \label{ineq:BkBk}
\end{align}
By Lemma \ref{lem:decompo} (1) we may choose $\tilde{\epsilon}$ sufficiently small depending on $\epsilon$ and $k$ sufficiently large to ensure that, for all $j$, 
\begin{equation}
K_\star := \min_{1 \leq j \leq N} \left( \frac{1}{m-1} - \frac{C_\star}{2 \bar{a}_j} (c_d + \tilde{\epsilon}) \norm{u_k|_{B_{k,j}}}_1^{2/d}\right) > c(a,\e)>0. \label{def:Kstar}
\end{equation}
This in turn fixes $K_4 := K(\tilde{\epsilon})$. 
Notice that $K_\star$ depends only on dimension, $a$ and $\epsilon$.

Applying \eqref{ineq:pekDone}, \eqref{ineq:BkE}, \eqref{ineq:BkBj}, \eqref{ineq:BkBk} and \eqref{def:Kstar} to \eqref{eq:PEdecomp} we have, 
\begin{align*}
\F(u_0) \geq \F(u_k) & = \frac{1}{m-1} \int u^m_k dx - PE(u_k,u_k) \\ 
& \geq \left(K_\star - K_1\delta^{1/d} - K_2\delta^{2/d}\right)\norm{u_k}_m^m - (N^2 + N + 2)C_E\norm{u_k}_m^{2\theta} - 2K(1) - K_3 N^2 - NK_4. 
\end{align*}
Since $K_\star$ depends only on $\epsilon$ and $K_1,K_2$ are fixed constants which do not depend on the decomposition, $\delta=\delta(a,\e)$ can be chosen a priori sufficiently small such that there exists a constant $\tilde{K}_\star > 0$,
\begin{align}
\F(u_0) \geq \tilde{K}_\star\norm{u_k}_m^m - (N^2 + N + 2)C_E\norm{u_k}_m^{2\theta} - 2K(1) - N^2 K_3 -NK_4. \label{ineq:FboundBelow}
\end{align}
Note that the choice of $\delta$ then fixes the decomposition given by Lemma \ref{lem:decompo} (and hence $N$).   
As $\F(u_0)$ is a fixed finite number and $2\theta < m$, this estimate implies that $\norm{u_k}_m$ is uniformly bounded for $k$ sufficiently large.
However, note that \eqref{ineq:FboundBelow} is generally vacuous unless $\norm{u_k}_m$ is extremely large. 
Regardless, \eqref{ineq:FboundBelow} implies that $\limsup_{k \rightarrow \infty}\norm{u_k}_m < \infty$ which, by the continuation criterion in Theorem \ref{thm:loc_theory}, implies that $u$ cannot blow-up at $T_+$. 
\end{proof}

We now prove Lemma \ref{lem:decompo}.  
\begin{proof} 
Let $u_k^l, e_k, u_k^V$ as given in Lemma~\ref{lem:CC} with $1\leq l\leq L$. For simplicity we only treat the case $L = 1$, as we may simply append resulting decompositions in order to treat $L > 1$. 
Denote $\tilde{n}_k^1 := u^1_k(x + x_k^1)$ and up to extraction of a subsequence, there exists some measure $n^\star$ such that $\tilde{n}_k^1 \rightharpoonup^\star n^\star$. 
Define the set $E_1 := \supp n^\star \subset B(0,R^1)$. 
By the assumption \eqref{def:epsilonnstar}, for every $x \in E_1$ there exists an $r_x > 0$ 
which satisfies: 
\begin{itemize}
\item[(a)] $n^\star(B(x,r_x)) < \left(\min_{ \abs{y-x} < r_x} a(y - x_k^l) \right)^{d/2}M_c^\star - \epsilon$,  
\item[(b)] $r_x$ is a point of continuity of the map $r \mapsto n^\star(B(x,r))$ (which is continuous a.e. since it is non-decreasing). 
\end{itemize} 
Denote this covering of $E_1$ by $\mathbf{B}^1$. 
Notice that in order to satisfy (a) given the assumption \eqref{def:epsilonnstar}, the radius of the balls must be chosen small enough such that $a(x)$ is sufficiently close to $a(y)$ for $\abs{x-y} < r_x$ (which requires the assumption of uniform modulus of continuity of $a(x)$). 
It now follows by the Besicovitch covering theorem (Theorem 2.7 \cite{Mattila}) that there exists a finite or countable 
set of balls $\mathcal{B}^1 \subset \mathbf{B}^1$ which satisfies:  
\begin{itemize}
\item[(i)] $\mathbf{1}_{E_1} \leq \sum_{B \in \mathcal{B}^1} \mathbf{1}_{B} \leq P(d),$ where $P(d)$ is a dimensional constant.
\item[(ii)] There exists $Q(d)$ families of disjoint balls $\mathcal{B}^1_1,...,\mathcal{B}^1_{Q(d)} \subset \mathcal{B}^1$ such that $E_1 \subset \cup_{i = 1}^{Q(d)}\cup \mathcal{B}^1_i$. Here again, $Q(d)$ is a dimensional constant.
\end{itemize}
Since the collections $\mathcal{B}^1_i$ define a covering of $E_1$ we have also 
\begin{equation*}
n^\star(E_1) \leq \sum_{i =1}^{Q(d)} \sum_{B \in \mathcal{B}^1_i} n^\star(B),    
\end{equation*}
which implies that one of the collections, without loss of generality suppose $\mathcal{B}^1_1$, satisfies
\begin{equation*}
\frac{n^\star(E_1)}{Q(d)} < \sum_{B \in \mathcal{B}^1_1} n^\star(B). 
\end{equation*}
We may truncate this sum to find a finite collection of $B \in \mathcal{B}^1_1$, labeled $B^1_{1,i}$ for $i \in \set{1,...,N_1}$ which instead satisfies
\begin{equation*}
\frac{n^\star(E_1)}{2Q(d)} \leq \sum_{i = 1}^{N_1} n^\star(B^1_{1,i}). 
\end{equation*} 
Now we describe how to carry the $m$-th step of the algorithm to the $m+1$-st and justify why the algorithm generates the decomposition stated in Lemma \ref{lem:decompo} in finitely many steps. 
Suppose we are given a set $E_{m+1} \subset B(0,R^1)$ which is defined by 
\begin{equation*}
E_{m+1} := E_1 \setminus \cup_{k = 1}^m \cup_{i = 1}^{N_k} B_{1,i}^k,  
\end{equation*}
where  we assume that the balls $\set{B_{1,i}^k}$, $1 \leq i \leq N_k$, $1 \leq k \leq m$ are disjoint and chosen such that 
\begin{equation}
n^\star(E_{m+1}) \leq \left(1 - \frac{1}{2Q(d)}\right)^{m}n^\star(E_1) = \left( 1- \frac{1}{2Q(d)} \right)n^\star(E_{m}). \label{ineq:geometric}
\end{equation}
From \eqref{ineq:geometric} it is clear that if $m$ is sufficiently large depending only on $n^\star(E_1)$, then $n^\star(E_{m+1}) < \delta$ and, in particular, 
\begin{equation*}
n^\star( \Real^d \setminus \cup_{k = 1}^m \cup_{i = 1}^{N_k} B_{1,i}^k ) = n^\star( E_{m+1} )  < \delta. 
\end{equation*}
Hence, the primary remaining step is to see that we really can progress from step $m$ to $m+1$. 
To this end, define a covering $\mathbf{B}^{m+1}$ of $E_{m+1}$ by balls $B(x,r_x)$ for all $x\in E_{m+1}$ which satisfy: 
\begin{itemize}
\item[(a)] $n^\star(B(x,r_x)) < \left(\min_{ \abs{y-x} < r_x} a(y + x_k^l) \right)^{d/2}M_c^\star - \epsilon$.
\item[(b)] $r_x$ is a point of continuity of the map $r \mapsto n^\star(B(x,r))$.
\item[(c)] $B(x,r_x) \cap \cup_{k = 1}^{m} \cup_{i = 1}^{N_k} B_{1,i}^k = \emptyset$. 
\end{itemize} 
To achieve the condition (c), generally $r_x$ must be taken small, however it is possible since the balls in question are closed. 
Given the covering $\mathbf{B}^{m+1}$, the Besicovitch covering theorem implies that we may choose a finite or countable subset $\mathcal{B}^{m+1} \subset \mathbf{B}^{m+1}$ which satisfies properties analogous to (i) and (ii) above, and moreover satisfies the additional property that $\mathcal{B}^{m+1}$ is disjoint from $B_{1,i}^k$.  
Given the sets chosen in (ii), we may again assume without loss of generality that, 
\begin{equation*}
\frac{n^\star(E_{m+1})}{Q(d)} \leq \sum_{B \in \mathcal{B}_1^{m+1}} n^\star(B),   
\end{equation*}
and furthermore that we may truncate this sum to be a finite number of balls labeled $B_{1,i}^{m+1}$ with $1 \leq i \leq N_{m+1}$ for some finite $N_{m+1}$ such that 
\begin{equation*}
\frac{n^\star(E_{m+1})}{2Q(d)} \leq \sum_{i = 1}^{N_{m+1}} n^\star(B_{1,i}^{m+1}).    
\end{equation*}
Appending $\set{B_{1,i}^{m+1}}$ to the existing list and defining $E_{m+2}$ by 
\begin{equation*}
E_{m+2} := E_1 \setminus \cup_{k = 1}^{m+1} \cup_{i = 1}^{N_k} B_{1,i}^k,  
\end{equation*}
we see that indeed we have advanced the algorithm to step $m+1$. 
By the argument above we saw that the algorithm terminates at some finite $M < \infty$. 
Hence define $N = \sum_{i = 1}^M N_i$ and re-index the balls $B_{1,i}^k$ to $\tilde{B}_j$, $1 \leq j \leq N$. 
Each ball $\tilde{B}_j$ has center $x_j \in B(0,R^1)$ with radius $r_j > 0$. We set $\hat{x}_k^j := x_j + x_k^1$.    
Since each $r_j$ is chosen as a point of continuity, (1) and (2) in the statement of the Lemma are satisfied by $u_k$ for $k$ sufficiently large. This can be verified by continuously approximating $\mathbf{1}_{B(x_j,r_j)}$ and using the fact that $\tilde{n}_k^1 \rightharpoonup^\star n^\star$. 
and $\limsup_{k \rightarrow \infty}\norm{e_k}_1 < \delta$.
Finally, observe that (3) in the statement of the Lemma is satisfied by construction and Lemma \ref{lem:CC}.
\end{proof}

\section{Finite Time Blow-Up} \label{sec:FTBU}
As discussed in the introduction, the inability to make obvious use of a virial method motivates our use of a barrier method based on maximum principle-type arguments.  We begin with the following rescaling: let $M_c < M_0 < M(u_0)$ be as in the statement of Proposition \ref{prop:InitialMassConcen}, define
 \begin{equation}
 \mu := M_c/M_0 < 1
 \end{equation}
  and let  
\begin{equation} 
\rho(t,x) := \mu u(\mu^{1-2/d}t,x). \label{def:resc}
\end{equation} 
Then $M(\rho) = M(u)\mu> M_c$ and $\rho$ solves 
\begin{equation} \label{def:ADD_rescaled} 
\left\{
\begin{array}{l}
  \rho_t + \grad \cdot (\rho  \mu^{1-2/d} \grad c) = \Delta \rho^{2-2/d} \\
  -\grad \cdot (a(x)\grad c) = \mu^{-1}\rho. 
\end{array}
\right.
\end{equation}

\vspace{10pt}

As mentioned above, we will use mass comparison arguments involving suitably chosen barriers (sub-solutions) to force the finite time concentration of mass in \eqref{def:ADD_rescaled}. 
We define the following comparison function $\bar{u} = \bar{u}(t,x)$ with $R_0$ as in Proposition \ref{prop:InitialMassConcen}, 
\begin{equation}
\bar{u}(t,x) = \frac{a(R_0)^{d/2}}{R(t)^d}V\left( \frac{x}{R(t)} \right),
\end{equation}
where we take $R(t)$ as a solution to the initial value problem
\begin{equation} \label{def:Reqn}
\left\{
\begin{array}{l}
  \dot{R}(t) = \frac{M_c(1- \mu^{-2/d})}{a(R_0)\sigma R(t)^{d-1}} \\
  R(0) = R_0.
\end{array}
\right.
\end{equation}
Here $\sigma$ denotes the surface area of the unit sphere.
Notice that $\dot{R} \leq 0$ and that $R(T_\star) = 0$ with
\begin{equation}
T_\star:=T_\star (M_0,R_0, d) = \frac{R_0^d\sigma}{da(R_0)M_c(\mu^{-2/d}-1)} < \infty. \label{sec4_blowuptime}
\end{equation}
We define the mass distributions 
\begin{equation*}
M(t,r) = \int_{\abs{x} \leq r} \rho(t,x) dx, \;\;\; \overline{M}(t,r) = \int_{\abs{x} \leq r} \bar{u}(t,x) dx.  
\end{equation*}
Notice that \eqref{ineq:massorderedic} is equivalent to $\overline{M}(r,0) \leq M(r,0)$, 
which means that the rescaled initial data is initially more concentrated than the barrier $\bar{u}$ on the neighborhood $r \leq R_0$.  
It is also important to note that the total mass of the barrier $\bar{u}$ is generally more than the critical mass $M_c$ but less than or equal to the total mass of $\rho$ which itself has less mass than the true solution $u$.  

Suppose that $\overline{M}(0,r)\leq M(0,r)$ for all $r\geq 0$. We will show that this ordering is preserved up to the blow-up time of $u$ or $\bar{u}$, e.g.
\begin{equation*}
\overline{M}(t,r) \leq M(t,r) \hbox{ for } 0\leq t<\min\left(T_\star, T_+(u_0)\right). \label{ineq:MassOrdered}
\end{equation*}
See Figure 2 for a graphical depiction of the relationship between $\bar{u}(t,x)$ and $u(t,x)$, which is very different from the way $M(t,r)$ and $\overline{M}(t,r)$ are related. 
\begin{center}
\begin{figure}[hbpt] \label{fig:Blowup}
\begin{picture}(150,150)(-100,0)
\put(60,45){$\bar{u}(t,x)$}
\put(90,120){$u(t,x)$}
\put(50,0){\scalebox{0.40}{\includegraphics{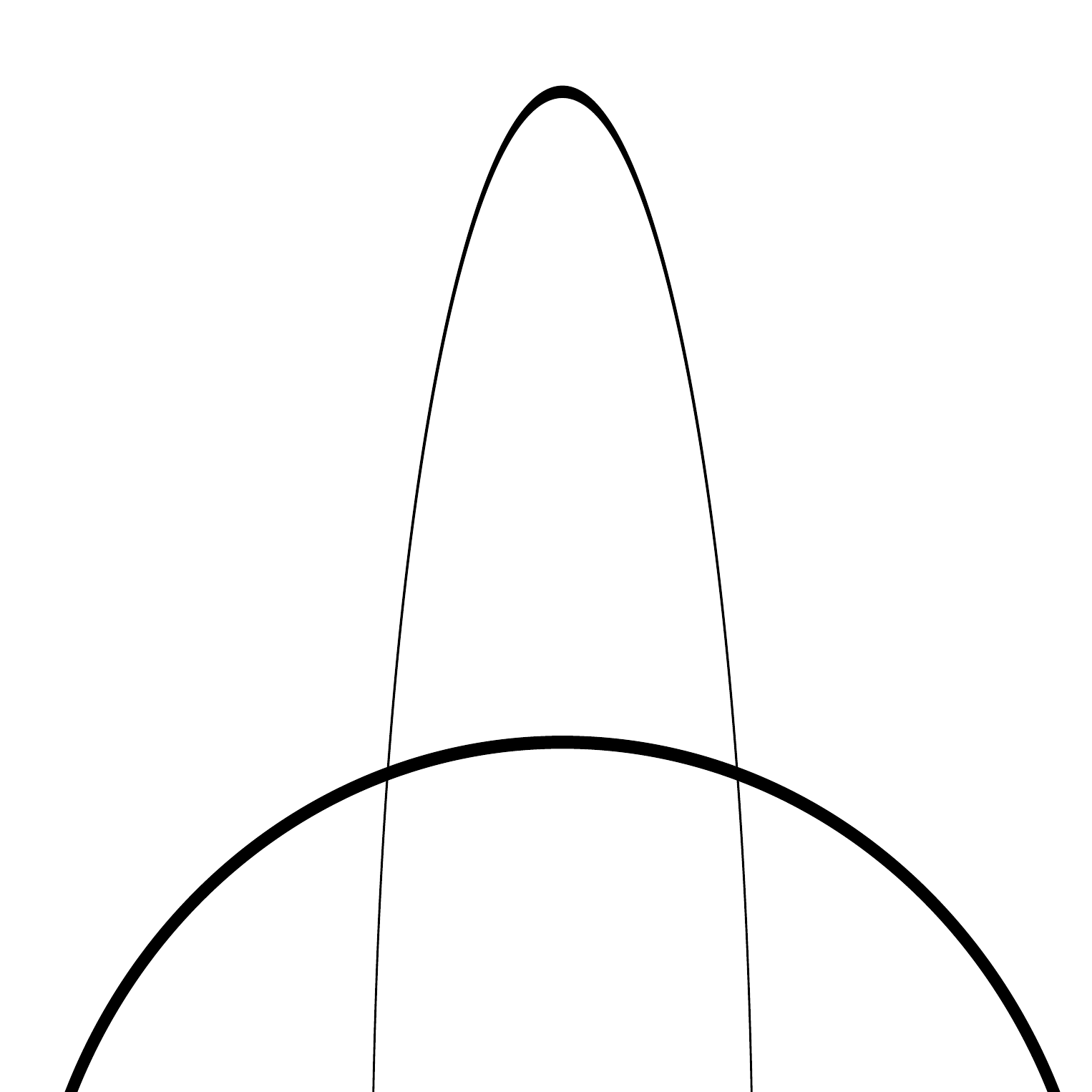}}}
\end{picture}
\caption{The mass subsolution $\bar{u}(t,x)$ and the real solution $u(t,x)$.}
\end{figure}
\end{center}

As alluded to above, in the language of maximum principle-type arguments, $\bar{u}$ plays the role of a {\it subsolution} in terms of the mass concentration.   As $\bar{u}$ concentrates into a delta mass at $t = T_\star$, we must have $T_+(u_0) \leq T_\star$, which will conclude the proof of Theorem \ref{thm:blowup}. 
The intuition for why the proof ultimately works is based on the fact that the rescaled system \eqref{def:ADD_rescaled} pulls mass into the origin faster than the PDE that $a^{d/2}(R_0)V$ solves (see Remark \ref{rmk:RescaledV}). 
That is, the rescaling \eqref{def:resc} transfers the property of having supercritical mass into surplus attractive power when compared against stationary solutions of roughly comparable mass. 
The surplus attractive power is what gives us the ability to choose $R(t)$ at the rate given in \eqref{def:Reqn} and hence prove that the solution is concentrating fast enough to be squeezed into blow-up by the self-similar barrier $\bar{u}$.    

\vspace{10pt}

We use mass comparison arguments influenced by those found in \cite{KimYao11} to prove that the mass ordering \eqref{ineq:MassOrdered} is preserved (see also \cite{BilerKarchEtAl06, JagerLuckhaus92, CieslakWinkler08}). As mentioned in the introduction, the novelty in our argument lies in the nature of our barriers which can handle near-critical mass accurately up till blow-up time. We should mention that, due to the inhomogeneity, the mass comparison principle that normally holds (in the sense of aforementioned references) does not apply.

\vspace{10pt}

A nontrivial complication arises at the free boundary of the positivity set of $u(t)$. Here, classical regularity breaks down and the mass comparison arguments no longer provide a rigorous argument.
To deal with this technical issue we lift to strictly positive solutions, which are smooth due to uniform parabolicity on compact sets.
Comparison with vanishing error is proved against these solutions, for which the formal arguments are rigorous. 
Passing to the limit requires some standard approximation arguments regarding the stability of \eqref{def:ADD_rescaled}.      
In order to prove positive solutions remain positive, we adopt viscosity solution-type arguments similar to those used in \cite{KL} (see Appendix for more information).  

\subsection{Preliminaries: Approximation and Regularity} 
In this section we detail the important approximation and regularity properties of \eqref{def:ADD}. 
These results are more or less expected, but since they are of independent interest and important for making our arguments rigorous, we include brief sketches of the proofs.

\begin{lemma}\label{lem:regularity1}
Suppose $u_0 \in C^0(\Real^d) \cap L_+^1(\Real^d;(1+\abs{x}^2)dx)$ and let $u(t)$ be the associated weak solution which satisfies $u(0) = u_0$.
 Then for all $\epsilon > 0$,  
\begin{itemize}
\item[(a)] $\grad c$ is continuous and bounded, and $\Delta c$ is bounded on $t \in [0,T_+(u_0) - \epsilon)$. 
\item[(b)]$u(t,x) \in C^{0}([0,T_+(u_0)-\epsilon)\times \Real^d)$.
\end{itemize}
\end{lemma}

\begin{proof}
\begin{itemize}
\item[(a)] The proof follows from standard elliptic regularity estimates and is omitted for brevity. 
\item[(b)] By (a), our PDE $u_t -\Delta u^{2-2/d}+\nabla\cdot(u\nabla c)=0$ can be viewed as a degenerate diffusion with a priori given drift term $c(x)$, which  satisfies the assumptions (A1)-(A3) of DiBenedetto \cite{DiBenedetto83}. In addition, from the proof of local existence theory (see  \cite{BRB10}), the approximation assumption (A4) is satisfied (i.e. $u$ can be approximated locally uniformly with the smooth solutions $u_n$ with strictly positive initial data). Therefore Theorem 1.1 of DiBenedetto \cite{DiBenedetto83} yields (b) (also see Theorem 3.1 of \cite{KimYao11}). 
\end{itemize}
\end{proof}

\begin{lemma}[Regularity of Strictly Positive Solutions] \label{lem:regularity2}
Let $u_0 \in L_+^1(\Real^d;(1+\abs{x}^2)dx) \cap L^\infty(\Real^d)$ be strictly positive a.e.. Then for all $\epsilon > 0$, the associated weak solution $u(t)$ with $u(0) = u_0$ 
is smooth and strictly positive on $(0,T_+(u_0)-\epsilon) \times \Real^d$. 
\end{lemma} 

\begin{proof}
Due to Lemma \ref{positivity} in the appendix, $u$ stays strictly positive for all $0\leq t< T_+(u_0)$. Consequently $u$ is a solution of a uniformly parabolic quasilinear PDE of divergence form, and the regularity of $u$ follows from classical regularity theory \cite{LadySoloUral}.
\end{proof}

\begin{lemma}[Stability of the blow-up time] \label{lem:full_approx}
Let $\set{u_0^n} \subset L_+^1(\Real^d;(1+\abs{x}^2)dx) \cap L^\infty(\Real^d)$ and $u_n(t)$ denote the associated weak solutions of \eqref{def:ADD} with $u_n(0) = u_0^n$ defined on the intervals $[0,T_+(u_0^n))$. 
Suppose further that
\begin{itemize}
\item[(a)] $\sup_{n} \left(\norm{u_0^n}_\infty + \norm{u_0^n}_1\right)  < \infty$ and  
\item[(b)] $u_0^n \rightarrow u_0$ strongly in $L^1$ for some $u_0 \in L_+^1(\Real^d;(1+\abs{x}^2)dx) \cap L^\infty(\Real^d)$. 
\end{itemize}

Let $u(t)$ be the solution to \eqref{def:ADD} with $u(0) = u_0$, and define $T_0 > 0$ such that,
\begin{equation*}
T_0 := \sup \set{ T \in (0,\infty) : \liminf_{n\rightarrow \infty} \sup_{t \in [0,T]}\norm{u_n(t)}_\infty < \infty}.     
\end{equation*}
Then for all $T < T_0$, there exists a subsequence such that $u_{n_k} \rightarrow u$ in $C([0,T],L^p(\Real^d))$ for all $1 \leq p <\infty$. 
If additionally $u_0^n \in C^0(\Real^d)$ and $u_0^n \rightarrow u_0$ locally uniformly, then we also have $u_{n_k} \rightarrow u$ locally uniformly on $[0,T] \times \Real^d$.
Moreover, $T_+(u_0) = T_0 \leq \liminf_{n \rightarrow \infty}T_+(u_0^n)$. 
\end{lemma} 

\begin{proof}
From the local existence theory and  the fact $\sup_n \left(\norm{u_0^n}_\infty + \norm{u_0^n}_1\right) < \infty$, it is assured that $T_0 > 0$ (see \cite{BRB10} or \cite{Blanchet09}).
Let $0 < T < T_0$. 
By the precompactness arguments of the local existence theory (see \cite{BRB10,BR11}) we may extract a subsequence $\set{u_{n_k}}$ which converges in $C([0,T];L^p(\Real^d))$ for all $1 \leq p < \infty$ to a weak solution, and by uniqueness of weak solutions, the limit must be $u(t)$.
In particular, we may extend $u(t)$ to include any time interval such that 
$$
\liminf_{n\rightarrow \infty} \sup_{t \in [0,T]}\norm{u_n(t)}_\infty < \infty.
$$
By the proof of Lemma \ref{lem:regularity1} and Theorem 6.1 of \cite{DiBenedetto83}, if $u_0^n \rightarrow u_0$ locally uniformly, it moreover follows that $u_{n_k}(t) \rightarrow u(t)$ locally uniformly up to extraction of an additional subsequence. 
By the continuation criterion in Theorem \ref{thm:loc_theory}, it is necessary that $T_0 \leq \liminf_{n \rightarrow \infty}T_+(u_0^n)$ 
but there is no a priori reason for them to be comparable (for instance, $T_0$ could be finite, but $T_+(u_0^n) \equiv \infty$). 
From the first part of the lemma, $u(t)$ exists on all compact time intervals of $[0,T_0)$.  Since $u_{n_k}(t) \rightarrow u(t)$ locally uniformly on this time interval, if $T_0 < \infty$, then we necessarily have $\liminf_{t \nearrow T_0} \norm{u(t)}_\infty = \infty$ and therefore $T_+(u_0) = T_0$. 
This proves the lemma. 
\end{proof} 

\subsection{Mass Comparison}
In this section we develop the comparison arguments and prove Theorem \ref{thm:blowup}.  

First, note that $\bar{u}(t)$ is a solution to the transport equation
\begin{equation}
\partial_t\bar{u} + \grad\cdot(\bar{u} \frac{\dot{R}}{R}x) = 0. \label{eq:ubar_transpt}
\end{equation}  
By construction, for all $t$, $\bar{u}(t,x)$ is also a weak solution to the scale-invariant problem (Remark \ref{rmk:RescaledV}), 
\begin{equation}
\frac{1}{a(R_0)}\grad \cdot (\bar{u} \grad \mathcal{N} \ast \bar{u}) = \Delta \bar{u}^m. \label{eq:ubar_quasistat}
\end{equation}

We have the following lemma which describes the PDE satisfied by the mass function corresponding to $\rho(t)$.
\begin{lemma}[Evolution of Mass Function] \label{lem:EvoMass}
Let $\rho(t,x)$ be a smooth radially symmetric solution to \eqref{def:ADD_rescaled}.
Then, $M(t,r) := \int_{\abs{x} \leq r} \rho(t,x) dx$ satisfies
\begin{align}
\partial_t M(t,r) = \sigma r^{d-1}\partial_r\left(\frac{\partial_rM(r)}{\sigma r^{d-1}}\right)^m  + \frac{\mu^{-2/d}}{a(r)}\frac{M(r)}{\sigma r^{d-1}}\partial_rM(r), \label{eq:Mpde}
\end{align}
where $\sigma$ is the surface area of the unit sphere in $\Real^d$. 
\end{lemma} 
\begin{proof}
By radial symmetry 
\begin{equation*}
\rho(t,r) = \frac{1}{\sigma r^{d-1}}\partial_r M(t,r)
\end{equation*}
and by the divergence theorem and the radial symmetry of $a$,
\begin{equation}
a(r)\int_{|x|=r}\partial_r c(t,x) dS = -\frac{\mu^{-1}}{\sigma r^{d-1}}M(t,r). \label{eq:NewtPull}
\end{equation}
Again using the divergence theorem and radial symmetry, 
\begin{align*}
\partial_t M(t,r) & = \int_{\abs{x} = r} \frac{x}{\abs{x}}\cdot \left( \grad \rho^m - \mu^{1-2/d}\rho\grad c \right) dS \\
& = \sigma r^{d-1}\partial_r\left( \frac{\partial_rM(t,r)}{\sigma r^{d-1}} \right)^m - \frac{\mu^{1-2/d}}{\sigma r^{d-1}}\partial_rM(t,r)\int_{\abs{x} = r} \partial_r c (t,r) dS \\ 
& = \sigma r^{d-1}\partial_r\left(\frac{\partial_rM(t,r)}{\sigma r^{d-1}}\right)^m  + \frac{\mu^{-2/d}}{a(r)}\frac{\partial_rM(t,r)}{\sigma r^{d-1}}M(t,r). 
\end{align*}
\end{proof}

Suppose $u_0$ satisfies \eqref{ineq:massorderedic} and let $\rho_\e(t)$ solve \eqref{def:ADD_rescaled} with initial data
\begin{equation*}
\rho_\e(0) = \rho_0 + \epsilon \frac{e^{-\abs{x}^2/4}}{(4\pi)^{d/2}}. 
\end{equation*}
By Lemma \ref{lem:regularity2}, $\rho_\e(t,r)$ remain smooth and positive on their time interval of existence $(0,T_+(\rho_\e(0)))$.  

We are now ready to state and prove the mass comparison result which will complete the proof of Theorem \ref{thm:blowup}. 

\begin{proposition} \label{prop:MassComp}
Suppose that $T < \min(T_\star,T_+(\rho_\e(0)))$ and let $M_\e(t,r) := \int_{\abs{x} \leq r}\rho_\e(t,x)dx$. Further suppose $\overline{M}(0,r) \leq M_\epsilon(0,r)$ for all $r\geq 0$. Then we have
\begin{equation*}
\overline{M}(t,r) - M_\epsilon(t,r) \leq 0 \hbox{ in }[0,T]\times \Real^d.
\end{equation*}

\end{proposition}
\begin{proof}
Similar to the proof of Lemma \ref{lem:EvoMass}, by \eqref{eq:ubar_quasistat} we have,
\begin{equation}
\sigma r^{d-1}\partial_r\left(\frac{\partial_r\overline{M}(r)}{\sigma r^{d-1}}\right)^m  + \frac{1}{a(R_0)}\frac{\overline{M}(r)}{\sigma r^{d-1}}\partial_r\overline{M}(r) = 0, \label{eq:Mbar_Stationary}
\end{equation}
and by \eqref{eq:ubar_transpt},
\begin{align}
\partial_t\overline{M}(r) & = -\int_{\abs{x} = r}\bar{u}(r)\frac{r}{R}\dot{R} dS \nonumber \\
& = -\frac{r}{R}\dot{R}\partial_r \overline{M}(r). \label{eq:Mbar_transpt}
\end{align}
For notational simplicity, define $M(t,r) := M_\epsilon(t,r)$. 

\vspace{10pt}

Consider the space-time region $(t,r) \in Q_T$, where
\begin{equation*}
Q_T = \set{(t,r): t\in[0,T], \;\; r \in [0,R(t)]}.  
\end{equation*}
As $M(\bar{u}) \equiv \overline{M}(t,r) \hbox{ in } |x|\geq R(t)$, we need only prove the comparison result in $Q_T$, from which the result on $(0,T)\times \Real^d$ follows since $M(t,r)$ is increasing in $r$. 

\vspace{10pt}

For a given constant $\lambda > 0$ (to be chosen later), let us consider the function 
$$
f(t,r):=(\overline{M}(t,r) - M(t,r))e^{-\lambda t}\hbox{ in }[0,T]\times \mathbb{R}^+.
$$
Note that $f(0,r)\leq 0$. If $f(t,r) \leq 0$ in $Q_T$ there is nothing to prove, so suppose that it is positive somewhere. 
Then $f(t,r)$ has a strictly positive maximum in $Q_T$, which is achieved at some point $(t_\star,r_\star)$. 
Necessarily, $r_\star > 0$. If $r_\star = R(t_\star)$ then by $r_\star$ being the location of the maximum, we must have
\begin{equation*}
0 = \partial_r\overline{M}(t_\star,r_\star) \geq \partial_rM(t_\star,r_\star) = \sigma r_*^{d-1}u^\e(t_\star, r_\star).
\end{equation*}
Since $u_\e$ is strictly positive, it follows that $r_\star < R(t_\star)$. 
This implies that due to maximization,
\begin{itemize}
\item[(A)] $\partial_t(\overline{M}(t_\star,r_\star) - M(t_\star,r_\star)) \geq \lambda(\overline{M}(t_\star,r_\star) - M(t_\star,r_\star))$.
\item[(B)] $\partial_r\overline{M}(t_\star,r_\star) = \partial_rM(t_\star,r_\star)$. 
\item[(C)] $\partial_{rr}M(t_\star,r_\star) \geq \partial_{rr}\overline{M}(t_\star,r_\star)$. 
\end{itemize}
Using the mass equations satisfied by each function we have, 
by \eqref{eq:Mpde}, \eqref{eq:Mbar_transpt} and \eqref{eq:Mbar_Stationary}, 
\begin{align*}
\partial_t(\overline{M} - M)(t_\star,r_\star) & = \sigma r_\star^{d-1}\partial_r\left(\frac{\partial_r\overline{M}}{\sigma r_\star^{d-1}}\right)^m  + \frac{1}{a(R_0)}\frac{\overline{M}}{\sigma r_\star^{d-1}}\partial_r\overline{M} \\
& \;\;\; - \sigma r_\star^{d-1}\partial_r\left(\frac{\partial_rM}{\sigma r_\star^{d-1}}\right)^m  - \frac{\mu^{-2/d}}{a(r_\star)}\frac{M}{\sigma r_\star^{d-1}}\partial_rM - \partial_r\overline{M}\frac{r_\star}{R(t_\star)}\dot{R}(t_\star).
\end{align*}
where all terms are evaluated at $(t_\star,r_\star)$.
By (B) and (C), we can order the higher order nonlinearity coming from the diffusion as well as relate the advection terms, 
\begin{align*}
\partial_t(\overline{M} - M)(t_\star,r_\star) & \leq \frac{\partial_r \overline{M}}{\sigma r_\star^{d-1}}\left(\frac{1}{a(R_0)}\overline{M} - \frac{\mu^{-2/d}}{a(r_\star)}M\right)
  - \partial_r\overline{M}\frac{r_\star}{R(t_\star)}\dot{R}(t_\star).
\end{align*}
Using that $r_\star \leq R_0$ and by assumption $a(r)$ is non-decreasing on $r \in [0,R_0]$ we have, 
\begin{align*}
\partial_t(\overline{M} - M)(t_\star,r_\star) & \leq \frac{1}{a(r_\star)}\frac{\partial_r \overline{M}}{\sigma r_\star^{d-1}}(\overline{M} - M) (t_\star, r_\star)\\ 
 & \;\;\; + \partial_r\overline{M}\left[\frac{(1 - \mu^{-2/d})}{a(r_\star)}\frac{M}{\sigma r_\star^{d-1}} - \frac{r_\star}{R(t_\star)}\dot{R}(t_\star)\right](t_\star, r_\star).
\end{align*}

Since $V$ is radial and non-increasing, we have 
$$
\partial_r(r^{1-d}\partial_r \overline{M}(t,r)) = \sigma^{-1}\partial_r(\bar{u}(t,r)) \leq 0.
$$
In addition we have $\overline{M}(t,0)=0$ and $\overline{M}(t,R(t))=M_c$. Note that  $h(r) := M_c (\frac{r}{R(t)})^d$ solves $\partial_r(r^{1-d}\partial_r h(r))=0$ in $(0,R(t))$ with boundary data $h(0)=0$ and $h(R(t))=M_c$. Therefore it follows from the weak elliptic maximum principle, 
\begin{equation}\label{observation}
\overline{M}(t,r) \geq M_c\left(\frac{r}{R(t)}\right)^d.
\end{equation}
Using the above observation we have
\begin{align*}
\partial_t(\overline{M} - M)(t_\star,r_\star) & \leq \frac{1}{a(r_\star)}\frac{\partial_r \overline{M}}{\sigma r_\star^{d-1}}(\overline{M} - M) (t_\star, r_\star) \\ 
 & \;\;\; + \frac{\partial_r\overline{M}}{r_\star^{d-1}}\left[\frac{(1 - \mu^{-2/d})}{a(r_\star)}\frac{M}{\sigma} -\frac{\overline{M}}{M_c}R(t_\star)^{d-1}\dot{R}(t_\star)\right](t_\star, r_\star).
\end{align*}
Due to \eqref{def:Reqn}, and using the fact that $a(r)$ is non-decreasing,  we have
\begin{align*}
\partial_t(\overline{M} - M)(t_\star,r_\star) & \leq \frac{1}{a(r_\star)}\frac{\partial_r \overline{M}}{\sigma r_\star^{d-1}}(\overline{M} - M) (t_\star, r_\star)\\ 
 & \;\;\; + \partial_r\overline{M}\left[\frac{(\mu^{-2/d}-1)}{a(r_\star)\sigma r_\star^{d-1}}(\overline{M}-M)\right](t_\star, r_\star).
\end{align*}
Since $\lambda(\overline{M}-M) \leq \partial_t(\overline{M}-M)$, the result follows by choosing
$$
\lambda > \sup_{t\in [0,T]}\|\bar{u}(t)\|_{\infty}\frac{1}{a(0)\mu^{2/d}}.
$$

\end{proof}

We may now prove Theorem \ref{thm:blowup}.

\begin{proof}(Theorem \ref{thm:blowup})
Let $u_0$ satisfy the hypotheses of Proposition \ref{prop:InitialMassConcen} and let $\set{\rho_\e}_{\e > 0}$ and $T_\star$ be given as above.
By the hypotheses of Proposition \ref{prop:InitialMassConcen},
$$
\overline{M}(0,r) \leq M(0,r) \leq M_\e(0,r)\hbox{ for }  r\in [0,\infty) \hbox{ for any }\e>0.
$$

Let $T_0$ be defined as in Lemma \ref{lem:full_approx}, which satisfies $T_+(\rho_0) = T_0 \leq \liminf_{\e \rightarrow 0} T_+(\rho_\e(0))$.  
Therefore, it suffices to show $T_0 \leq T_\star$. 

\vspace{10pt}

To this end, suppose $T_0 > T_\star$, which implies that 
$\set{\rho_\e(t)}$ and $\rho(t)$ exist on $[0,T_\star]$ for sufficiently small $\e$. 
Moreover, by Lemma \ref{lem:full_approx}, there exists a sequence $\rho_{\e_k} \rightarrow \rho$ in $C([0,T_\star];L^1(\Real^d))$ and locally uniformly.
Combined with Proposition \ref{prop:MassComp}, this implies 
\begin{equation}
\overline{M}(t,r) \leq M(t,r), \label{ineq:LocalMassOrdered} 
\end{equation}          
for all $0 \leq t < T_\star$. 
Since $\bar{u}$ concentrates at time $T_\star$, \eqref{ineq:LocalMassOrdered} implies that $\rho(t)$ must also concentrate at $T_\star$, contradicting the assumption $T_+(u_0) = T_0  > T_\star$.
\end{proof}

We briefly sketch a proof of Theorem \ref{cor:HomogProblem}. 
\begin{proof}(Theorem \ref{cor:HomogProblem})
Let $u_0$ be as in the statement of Theorem \ref{cor:HomogProblem}. We only need to verify \eqref{ineq:massorderedic} in order to apply Proposition \ref{prop:InitialMassConcen}. 
Let $R_1$ be such that $M_0 = \int_{\abs{x} \leq R_1} u_0 dx > M_c$. Then \eqref{ineq:massorderedic} holds for any $r > R_1$.
By assumption, $M_c M_0^{-1}u_0$ is strictly positive on some compact ball $\set{\abs{x} \leq r_1}$. Hence we may choose $R_0$ 
sufficiently large such that $R_0^{-d}V(R_0^{-1}x) < M_c M_0^{-1}u_0(x)$ for $\abs{x} \leq r_1$ and therefore \eqref{ineq:massorderedic} holds up to at least $r = r_1$. As for $r_1 \leq r \leq R_1$, $M_c M_0^{-1}\int_{\abs{x} \leq r}u_0(x) dx$ is non-decreasing and hence bounded below on the compact annulus by the value at $r = r_1$. Hence, we may choose $R_0$ even larger to ensure that \eqref{ineq:massorderedic} holds also for $r_1 \leq r \leq R_1$ and therefore everywhere. Hence, we may apply Proposition \ref{prop:InitialMassConcen} and the result follows.   

We now prove that we may construct blow up solutions with arbitrarily large initial free energy. 
We follow a similar procedure as Lemma 3.7 in \cite{Blanchet09}.  
Let $\mathcal{V}_M \subset L^1_+ \cap L^m$ be the set of non-negative, radially symmetric non-increasing functions in $L^1 \cap L^m$ with mass $M$. By the above reasoning, if $u_0 \in \mathcal{V}_M$ is continuous with finite second moment then the associated solution $u(t)$  to the scale-invariant problem with $u(0) = u_0$ blows up in finite time. Now we prove that 
\begin{equation*}
\sup_{h \in \mathcal{V}_M} \F(h) = +\infty. 
\end{equation*}
Suppose for contradiction that   
\begin{equation}
A := \sup_{f \in \mathcal{V}_M} \F(h) < \infty. \label{ineq:Ffinite} 
\end{equation}
Following the same scaling argument as in Lemma 3.7 of \cite{Blanchet09}, we may use the HLS \eqref{ineq:NHLS} to show \eqref{ineq:Ffinite} implies a reverse H\"older-type inequality for any $h \in L^1_+ \cap L^m$ which is radially symmetric non-increasing 
\begin{equation*}
\norm{h}_m^m\norm{h}_1^{2/d} \lesssim_M \norm{h}_{2d/(d+2)}^2. 
\end{equation*}
However, this inequality is definitely false, as $m > 2d/(d+2)$ implies we may easily construct a sequence of functions with uniformly bounded $L^{2d/(d+2)}$ norm and unbounded $L^m$ norm. Indeed, consider $f_\delta = (\delta + \abs{x})^{-\alpha}\mathbf{1}_{B(0,1)}(x)$ with $\alpha \in (d/m,(d+2)/2)$ with $\delta \in [0,1]$. Then $\norm{f_\delta}_1 \geq \norm{f_1}_1 > 0$ but $\norm{f_\delta}_{2d/(d+2)}$ is uniformly bounded and $\lim_{\delta \rightarrow 0}\norm{f_\delta}_{m} = \infty$.  
Hence it follows by contradiction that $A = +\infty$. By density, we may restrict to continuous functions with finite second moment in $\mathcal{V}_M$ and show that there are solutions to the scale-invariant problem with initial free energy arbitrarily large which blow up in finite time.  
\end{proof} 

\section{Global existence at critical mass} \label{sec:GECritMass}
In the scale-invariant case, the free energy no longer controls the entropy if $M(u) = M_c$. Indeed, consider a family of extremals (from Theorem \ref{thm:V}), $V_R = R^{-d}V(R^{-1} x)$ with $R \rightarrow 0$, which concentrates into a Dirac mass but all satisfy $\F(V_R) \equiv 0$. 
Hence unlike the subcritical case, a proof of global existence at the critical mass requires more information than what is provided by the energy dissipation inequality alone.

In the homogeneous case, the mass comparison principles with a supersolution prove global existence of radial threshold solutions \cite{Yao11}, a result which extends to general threshold solutions by symmetrization inequalities \cite{DiazNagai95,DiazNagaiRakotoson98,KimYao11}. 
However, it is important to note that \eqref{def:ADD} does \emph{not} satisfy the same comparison principles as the homogeneous problem and in particular, radial solutions which are initially monotone are not guaranteed to remain so. 
Among other things, this implies that a mass supersolution alone cannot be used to imply global existence. 
However, Theorem \ref{thm:Concentration} provides a kind of rigidity to threshold blow-up solutions, as all of the mass in the solution must concentrate into a single point where $a(x)$ achieves the minimal value. It is this additional information which allows us to prove Theorem \ref{thm:GE_CritMass}, as now we only need to rule out a very specific behavior.  
Indeed, we use a mass supersolution to show that any radially symmetric threshold solution cannot concentrate into the origin, even if the blow-up is in infinite time. 
Therefore, by Theorem \ref{thm:Concentration}, all radially symmetric critical mass solutions must remain uniformly bounded.   

 Certain details of the mass comparison argument used to prove Theorem \ref{thm:GE_CritMass} differ from those used to prove Theorem \ref{thm:blowup}. In fact, the proof of Theorem \ref{thm:GE_CritMass} is significantly easier.  
In this section we use the extremals as mass supersolutions, as opposed to the previous section where our barriers were subsolutions.
Define 
\begin{equation*}
\bar{a} := \min_{x \in \Real^d} a(x), 
\end{equation*}
and our supersolution to be the extremal 
\begin{equation}
\bar{u}(t,x) = \frac{\bar{a}^{d/2}}{R^d}V\left( \frac{x}{R} \right),
\end{equation}
where now $R > 0$ is a sufficiently small constant which will be fixed later in the proof. Note as well that $M(\bar{u}) = M_c$. 
As above in \S\ref{sec:FTBU}, we define the mass distributions 
\begin{equation*}
M(t,r) = \int_{\abs{x} \leq r} u(t,x) dx, \;\;\; \overline{M}(t,r) = \int_{\abs{x} \leq r} \bar{u}(t,x) dx.  
\end{equation*}
In contrast to \S\ref{sec:FTBU}, $R$ is chosen such that $M(0,r) \leq \overline{M}(r,0)$ for all $r \in [0,R]$, which is always possible since $u_0$ is continuous. 
We now show the following proposition. Since the proof is similar to the proof of Proposition \ref{prop:MassComp} we only sketch it. 
In light of Theorem \ref{thm:Concentration}, this proposition implies Theorem \ref{thm:GE_CritMass}, as it rules out the finite or infinite time concentration of mass.

\begin{proposition}\label{prop:CritMassComp}
Suppose that $T < \min(T_+(u_0))$. Further suppose $M(0,r) \leq \overline{M}(0,r)$ for all $r\geq 0$. Then we have
\begin{equation}
M(t,r) - \overline{M}(t,r) \leq 0 \hbox{ in }[0,T]\times \Real^d. \label{ineq:GE_MassComp}
\end{equation}
\end{proposition} 
\begin{proof} 
As in the previous section, we must regularize the initial data. Hence, fix $\epsilon>0$, define 
\begin{equation*}
u_0^\epsilon(x) =  \frac{1}{4\pi \epsilon} \int e^{-\frac{\abs{x-y}^2}{4\epsilon}} u_0(y) dy. 
\end{equation*}
Note that we use a different regularization here than in the proof of Proposition \ref{prop:MassComp} in order to preserve the mass.
Denote the associated weak solutions $u^\epsilon(t)$, which are smooth and strictly positive until blow-up time $T_+(u_0^\epsilon)$ by Lemma \ref{lem:regularity2}.  
We will deduce \eqref{ineq:GE_MassComp} independent of $\epsilon$ and hence Lemma \ref{lem:full_approx} shows that the inequality holds also for $u(t)$. 

The proof that \eqref{ineq:GE_MassComp} is satisfied independent of $\epsilon$ is a simple variation on the proof of Proposition \ref{prop:MassComp}, made significantly more straightforward by the simpler barrier $\bar{u}(x,t)$ which here does not depend on time.
The difference is that $\bar{u}(x,t)$ plays the role of supersolution here, and so we are looking for the positive maximum of $ (M(t,r) - \overline{M}(t,r))e^{-\lambda t}$.
The main place the argument differs significantly from Proposition \ref{prop:MassComp} is along the contact line $\abs{x} = R$ at the edge of the support of $\bar{u}$. 
We can rule out that a positive maximum occurs here for the simple reason that the barrier and the weak solutions have the same mass. Hence the solutions $u^\epsilon$ cannot possibly have strictly more mass in a ball of $\abs{x} \leq R$ than the barrier.   
Therefore, any positive maximum must occur somewhere in the set $0 < \abs{x} < R$. From this point, we may continue as in the proof of Proposition \ref{prop:MassComp}. Note that unlike the proof of Proposition \ref{prop:MassComp}, the presence of $\gamma(x)$ does not pose a problem here. 
\end{proof} 

\begin{remark}
Theorem \ref{thm:Concentration} shows that any solution with critical mass can only blow up by concentrating all of the mass into a point where $a(x)$ achieves the minimum value.
Hence with the additional constraint of radial symmetry, Theorem \ref{thm:Concentration} directly implies Theorem \ref{thm:GE_CritMass} without the need for Proposition \ref{prop:CritMassComp} unless $\min a(x) = a(0)$.  
\end{remark}

\section*{Acknowledgments} 
The authors would both like to thank Yao Yao and Nancy Rodr\'{i}guez for their many very helpful discussions and suggestions, Nader Masmoudi and Jonas Azzam for helpful discussions regarding the proof of Theorem \ref{thm:Concentration} and Guillaume Bal for the interesting discussions which prompted this research.  
J. Bedrossian was partially supported by NSF Postdoctoral Fellowship in Mathematical Sciences DMS-1103765 and NSF grant DMS-0907931 and I. Kim was partially supported by NSF grant DMS-0970072.    

\section*{Appendix: Pressure Form Comparison}
By Lemma \ref{lem:regularity1}, for any $\e>0$ we have
\begin{equation}\label{theconstant}
L:= \sup_{\epsilon \leq t\leq T^*-\epsilon} |\Delta c| <\infty.
\end{equation}
Let us define the pressure form of $u$:
\begin{equation}\label{pressure}
v=\frac{m}{m-1} (u)^{m-1}.
\end{equation}
Then formally $v$ solves the following equation:
$$
v_t = (m-1)v\Delta v +|Dv|^2 +\nabla v \cdot \nabla c + (m-1) v\Delta c.
$$
 
 \vspace{5pt}
 
 We proceed to prove the``viscosity solution" property of the pressure $v$. The notion of viscosity solutions are first introduced for Hamilton-Jacobi equations by Crandall-Lions (\cite{CL}) and later for fully nonlinear elliptic-parabolic equations (\cite{CIL}) as well as free boundary type problems (see \cite{BV}, \cite{CV} and \cite{KL} for porous medium-type problems.) The advantage of the approach lies in pointwise control of solutions and, in our setting, their free boundaries. More specifically we will show that the initially positive solutions cannot touch down to zero at later times, i.e. that contact lines cannot be nucleated.

\vspace{10pt}

Since $c$ is not $C^2$ up to the zero set of $u$, $\Delta c$ is not well-defined on the free boundary \\$\partial\{u>0\}=\partial\{v>0\}$.  This causes a technical problem for directly applying a standard notion of viscosity solutions to $v$. Hence we will directly prove the necessary properties to be used in our analysis in the next section.

\vspace{10pt}

\begin{definition} \label{def:Crossings}
For nonnegative functions $u$ and $v$ defined in a small neighborhood $\Sigma$ of $(x_0,t_0)$, we say
\begin{itemize}
\item[(a)] $u$ {\it crosses} $v$ {\it from below} at $(x_0,t_0)$ if 
$$
u \leq v \hbox{ in }\Sigma\cap\{t\leq t_0\}\hbox{ and } u(x_0,t_0)= v(x_0,t_0).
$$
\item[(b)] $u$ {\it crosses} $v$ {\it from above} at $(x_0,t_0)$ if 
$$
u \geq v \hbox{ in }\Sigma\cap\{t\leq t_0\}\hbox{ and } u(x_0,t_0)= v(x_0,t_0).
$$\\
\end{itemize}

\end{definition}

\vspace{10pt}

\begin{proposition}\label{viscosity}
For any given domain $\Sigma\subset\Real^d\times [0,T^*-\e)$, let $\phi$ be a nonnegative continuous function in $\Sigma$  which is $C^{2,1}$ in $\overline{\{\phi>0\}}$, with $|D\phi|>0$ on $\partial\{\phi>0\}$. 
Let $v$ be the pressure form of $u$ as defined in \eqref{pressure}. Then the following holds:
\begin{itemize}
\item[(a)] Suppose $v$ crosses $\phi$ from below at $(x_0,t_0)$ in $\overline{\{u>0\}}\cap\{t\leq t_0\}$ in $\Sigma$. Then we have
$$
\phi_t -(m-1)\phi\Delta\phi - |D\phi|^2-\nabla\phi\cdot\nabla c -(m-1)L\phi \leq 0
$$
\\
\item[(b)] Suppose $v$ crosses $\phi$ from above at $(x_0,t_0)$ in $\overline{\{u>0\}}\cap\{t\leq t_0\}$ in $\Sigma$. Then 
we have 
$$
\phi_t -(m-1)\phi\Delta\phi-|D\phi|^2-\nabla\phi\cdot\nabla c +(m-1)L\phi \geq 0.
$$ 
\end{itemize}
Here the constant $L$ is as given in \eqref{theconstant}.
\end{proposition}

\begin{proof}
1. Note that $u$ and $v$ are smooth in their positive set, and there the result follows easily. Hence, the only difficult case is when $(x_0,t_0)\in\partial\{v>0\}$.

2.  Let us take $c_\e:= c \ast\eta_\e$, where $\eta_\e$ is the standard mollifier. Let $u_\e$ be the weak solution of 
$$
(u_\e)_t = \Delta(u_\e)^m + \nabla\cdot(u_\e \nabla c_\e),
$$
with initial data $u_\e(x,0)=u(x,0)$. Since $c_\e$ is $C^2$, it follows from \cite{KL} that $u_\e$ is a viscosity solution in the sense defined in therein. In particular, it is shown in \cite{KL} that the statements in the proposition hold for $v_\e$: the pressure form of $u_\e$.
Below we will approximate $u$ by $u_\e$ to prove the proposition. Since $c_\e$ is uniformly bounded in $C^{1,1}$ norm, $u_\e$ is equi-continuous due to Theorem 6.1 of \cite{DiBenedetto83}. Using this fact, parallel arguments leading to Proposition 3.3 of \cite{KimYao11} yields that $u_\e$ uniformly converges to  $u$, and thus $v_\e$ to $v$.

3. Let us now show (a) when $v$ crosses a nonnegative function $\phi\in C^{2,1}(\overline{\{\phi>0\}})$ from below at $(x_0,t_0)\in\partial\{v>0\}$. Let us perturb $\phi$ so that $v$ really crosses $\phi$, not just touching. This can be done by replacing $\phi$ with
$$
\tilde{\phi}(x,t) = (\phi(x,t) -a(t-t_0-b))_+,
$$
where $a$ and $b$ are  small positive constants.

4. If (a) fails then, since $\phi(x_0,t_0)=0$, then $\phi$ satisfies 
$$
(\phi_t-(m-1)\phi\Delta\phi-|D\phi|^2-\nabla\phi\cdot\nabla c - (m-1)L\phi )(x_0,t_0) > 0.
$$

Now let us pick a small $\delta>0$ and take $\phi_\delta(\cdot,t):=(\phi_+) (\cdot,t)* \eta_{\delta}+m(\delta)$ where $\eta(x)$ is a standard mollifier which is smooth and has exponential decay at infinity, and $m(\delta)$ is a constant.
Choose $m(\delta)$ accordingly so that  $v$ is strictly below $\phi_{\delta}$ at $t=t_0$ but crosses $\phi_{\delta}$ from below at $t=t_0+O(\delta)$. Note that this is possible because $\eta$ does not have a compact support.

\vspace{10pt}

Then  due to continuity  of the derivatives of $\phi$ in its support and the corresponding convergence of $\phi_{\delta}$ to $\phi$, $\phi_{\delta}$ satisfies 
\begin{equation}\label{contradiction}
(\phi_{\delta})_t - (m-1)\phi_{\delta}\Delta\phi_{\delta}-|D\phi_{\delta}|^2 - \nabla\phi_{\delta}\cdot \nabla c_\e  -(m-1)L\phi_{\delta} >0.
\end{equation}
in $O(\delta_0)$-neighborhood of $(x_0,t_0)$ if $\e,\delta<<\delta_0$.

\vspace{10pt}

Since $v_\e$ converges uniformly to $v$ as $\e\to 0$,  $v_\e$ crosses $\phi_{\delta}$ from below at $(x_\delta,t_\delta)$, which lies in $O(\delta_0)$-neighborhood of $(x_0,t_0)$ if $\e$ and $\delta$ are chosen small enough. 

Note that at $(x_0,t_0)$ we have
$$
(v_\e)_t \geq (D\phi_{\delta})_t, \quad |Dv_\e| = |D\phi_{\delta}|\hbox{ and } \Delta v_\e \leq \Delta \phi_{\delta}.
$$
 this contradicts \eqref{contradiction} and the fact that $v_\e$ satisfies
$$
(v_\e)_t - (m-1)v_\e\Delta v_\e-|Dv_\e|^2 - \nabla v_\e\cdot \nabla c_\e - (m-1)Lv_\e \leq 0
$$

\end{proof}

\begin{corollary}[Local comparison in pressure variable] \label{cor:LocComp}
In any given parabolic, cylindrical neighborhood $\Sigma$, let $\phi$ be a $C^{2,1}$ function in $\overline{\{\phi>0\}}$ with $|D\Phi|>0$ on $\partial\{\phi>0\}$.
\begin{itemize}
\item[(a)] Suppose that $\phi$ satisfies
$$
\phi_t-(m-1)\phi\Delta\phi-|D\phi|^2-\nabla\phi\cdot\nabla c -(m-1)L\phi > 0 \hbox{ in } \Sigma.
$$
Then $v$ cannot cross $\phi$ from below in $\Sigma$.

\item[(b)] Suppose $\phi$ satisfies
$$
\phi_t-(m-1)\phi\Delta\phi-|D\phi|^2-\nabla\phi\cdot\nabla c +(m-1)L\phi < 0 \hbox{ in } \Sigma.
$$
Then $v$ cannot cross $\phi$ from above in $\Sigma$.
\end{itemize}
\end{corollary}

An immediate consequence of the above proposition is the preservation of positivity for $u$.

\begin{lemma}\label{positivity}
Let $u(x,t)$ be the weak solution associated with the strictly positive continuous initial data $u_0(x)>0$. Then $u(x,t)$ is strictly positive everywhere up to the blow-up time $T_+(u_0)$.
\end{lemma}

\begin{proof}
Let $v$ be the corresponding pressure form of $u$.  Let us recall that the Barenblatt profile is given as 
$$
B(x,t):= t^{-\lambda}\left(C-k\frac{|x|^2}{t^{2\mu}}\right)_+
$$
where  $C>0$ is a positive constant and 
$$
\lambda=\dfrac{d(m-1)}{d(m-1)+2}, \mu = \frac{\lambda}{d}, k= \frac{\lambda}{2d}.
$$
$B(x,t)$ then solves the porous medium equation in its pressure form in the viscosity sense (see e.g. \cite{KL}):
$$
B_t - (m-1)B\Delta B -|DB|^2=0.
$$
Let us now define 
$$
\tilde{B}(x,t) : = e^{-Mt}\sup_{y\in B_{M-Mt}(x)} B(x,t) \hbox{ for } 0\leq t\leq 1.
$$
Then due to  Proposition 2.13  in \cite{KL}  $\tilde{B}$ satisfies 
$$
\tilde{B}_t - (m-1)\tilde{B}\Delta\tilde{B} -|D\tilde{B}|^2+M|D\tilde{B}|+M\tilde{B} \leq 0
$$
for $0\leq t\leq 1$.  Let us choose 
$$
M = (m-1)\max\left(\|\nabla c \|_{L^\infty} , L\right).
$$ 
(Note that the first term in above upper bound is bounded before the blow-up time).

\vspace{10pt}

Since $B(x,t)$ vanishes uniformly to zero as $C\to 0$, so does $\tilde{B}$. Hence for any $\tau>0$ one can choose $C=C(\tau)$ sufficiently small so that $\tilde{B}(x, \tau) \leq u_0$. Then Corollary~\ref{cor:LocComp} yields that
$$
\tilde{B}(x,t+\tau) \leq u(x,t)\hbox{ for } 0\leq t\leq 1.
$$

Since $\tau$ can be arbitrarily large and $\tilde{B}$ has its support expanding to the whole domain as $\tau$ grows to infinity, we conclude that $u$ is strictly positive for $0\leq t\leq 1$.

\vspace{10pt}

We can iterate above argument up to the blow-up time to conclude (since the solution and $L$ remain bounded until blow-up).  

\end{proof}

\vfill\eject
\bibliographystyle{plain}
\bibliography{nonlocal_eqns,dispersive}

\end{document}